\documentclass[oneside,english,final,11pt]{article} 

\usepackage[utf8]{inputenc} 
\usepackage[usenames,dvipsnames]{color} 
\usepackage[T1]{fontenc} 
\usepackage{amssymb,amsthm,amsmath} 
\usepackage{babel}
\usepackage{svn}
\usepackage{graphicx}
\usepackage{subfig}
\usepackage[notref,notcite]{showkeys}
\usepackage{ifdraft}
\usepackage{geometry}
\usepackage[unicode=true,
 bookmarks=false,
 breaklinks=false,pdfborder={0 0 1},backref=false,colorlinks=false]
 {hyperref}

\usepackage[usenames]{color}
\definecolor{Red}{rgb}{0.9,0.1,0.3}
\newcommand{\ev}[1]{\mathbb{E}{#1}}
\newcommand{\e}[1]{\mathbb{E}}
\newcommand{\pr}[1]{\mathbb{P}\rbr{#1}}
\newcommand{\norm}[3]{\Vert #1 \Vert_{ #2 } ^{ #3 }}
\newcommand{\rbr}[1]{\left( #1 \right)}
\newcommand{\sbr}[1]{\left[ #1 \right]}
\newcommand{\cbr}[1]{\left\{ #1 \right\}}
\newcommand{\ddp}[2]{\left\langle #1, #2 \right\rangle}
\newcommand{\intr}{\int_{\mathbb{R}^d}}

\newcommand{\intc}[1]{\int_{0}^{#1}}
\newcommand{\Rd}{\mathbb{R}^d}

\newcommand{\R}{\mathbb{R}}
\newcommand{\T}[1]{\mathcal{T}_{#1}}

\definecolor{czerwony}{rgb}{1,0.3,0.3}
\definecolor{zielony}{rgb}{0.1,0.8,0.4}
\ifdraft{\newcommand{\TDD}[1]{\textcolor{czerwony}{ToDo #1} }} {\newcommand{\TDD}[1]{ }}

\newcommand{\dd}[1]{\textnormal{d}#1}

\newcommand{\SP}{\mathcal{S}'(\Rd)}
\newcommand{\SD}{\mathcal{S}(\Rd)}

\newcommand{\lap}[1]{\ev{\exp\rbr{#1}}}

\newcommand{\ti}{\rightarrow +\infty}

\newcommand{\tv}{\tilde{v}_T} 
\newcommand{\eexp}[1]{\exp \left( #1 \right)}

\theoremstyle{plain} 
\newtheorem{thm}{Theorem}[section]

\newtheorem{lem}{Lemma}[section]

\newtheorem{corollary}{Corollary}[section]
\newtheorem{prop}{Proposition}[section]

\theoremstyle{remark}\newtheorem{rem}{Remark}[section]
\newtheorem*{acknowledgement*}{Acknowledgement}

\title{Occupation times of subcritical branching immigration systems with Markov motion, clt and deviation principles} 
\author{\textsc{Piotr Miłoś\footnote{supported by MNISW grant N N201 397537.}}} 
\begin{document}
	\maketitle

	\begin{abstract}
		In this paper we consider two related stochastic models. The first one is a branching system consisting of particles moving according to a Markov family in $\Rd$ and undergoing subcritical branching with a constant rate of $V>0$. New particles immigrate to the system according to a homogeneous space–time Poisson random field. The second model is the superprocess corresponding to the branching particle system. We study rescaled occupation time process and the process of its fluctuations under mild assumptions on the Markov family. In the general setting a functional central limit theorem as well as large and moderate deviation principles are proved. The subcriticality of the branching law determines the behaviour in large time scales and it ``overwhelms'' the properties of the particles’ motion. For this reason the results are the same for all dimensions and can be obtained for a wide class of Markov processes (both properties are unusual for systems with critical branching).\\
				
		MSC: primary 60F17; 60G20; secondary 60G15 \\
		Keywords: Functional central limit theorem; Occupation time fluctuations; Branching particles systems with immigration; Subcritical branching law
		
	\end{abstract}
	
	\section{Introduction} In this paper we study two closely related random models. The first one is a subcritical branching particle system (BPS) with immigration. It consists of particles evolving independently in $\Rd$ according to a time-homogeneous Markov family $(\eta_t,\mathbb{P}_x)_{t\geq0,x\in \Rd}$. The lifetime of a particle is distributed  exponentially with a parameter $V>0$. When dying the particle splits according to a binary branching law, determined by the generating function 
	\begin{equation}
		F(s) = qs^2 + (1-q),\qquad q<1/2. \label{eq:generating} 
	\end{equation}
	This branching law is \textit{subcritical} (i.e. the expected number of particles spawning from one is strictly less than $1$). Each of the new-born particles undertakes movement according to the Markov family $\eta$, independently of the others, branches, and so on. New particles \textit{immigrate} to the system according to a homogeneous Poisson random field in $\mathbb{R}_+\times \Rd$ (i.e. time and space) with the intensity measure $H \lambda_{d+1}$, $H>0$ (where $\lambda_{d+1}$ denotes the $(d+1)$-dimensional Lebesgue measure). Because of immigration the initial particle distribution has no effect on the system in the long term. For the sake of simplicity, we choose it to be null.\\
	The second model considered in the paper is the superprocess corresponding to the BPS. It can be obtained as a short life-time, high-density, small-particle limit of the BPS described above. This construction is standard and recalled in Section \ref{sec:approx}. The evolution of these models will be represented by empirical measure processes $(N^B_t)_{t \geq 0}$, $(N^S_t)_{t \geq 0}$ (for the BPS and the superprocess respectively); i.e. for a Borel set $A$, $N^B_t(A)$ ($N^S_t(A)$) denotes a random number of particles (random mass) in $A$ at time $t$. We will also use the shorthand $N$ when we speak about both models. We define the rescaled occupation time process $(Y_T(t))_{t\geq 0}$ by
	\begin{equation}
		Y_T(t) := \frac{1}{F_T}\intc{Tt} N_s \dd{s},\: t\geq 0,\label{eq:rescaled-occupation} 
	\end{equation}
	and its fluctuations $(X_T(t))_{t\geq 0}$ by
	\begin{equation}
		X_T(t) := \frac{1}{F_T} \intc{Tt} \left( N_s - \ev N_s \right) \dd{s}, \: t \geq 0. \label{def:occupation-process} 
	\end{equation}	
In both cases $F_T$ is a deterministic norming which may vary in different situations. \\
We will now discuss the results obtained in the paper. The behaviour of both models is very similar hence will be presented together. Informally speaking, when $F_T = T$ the following \emph{law of large numbers} holds - $Y_T(t) \rightarrow_T \mu t$, for a certain positive measure $\mu$. Our aim is to estimate the speed of this convergence. This will be done through the following: 
\begin{description}
	\item[Central limit theorems (CLT)] The objectives of this part are to find suitable $F_T$, such that $X_T$ converges in law as $T\rightarrow +\infty$ to a non-trivial limit and identify this limit. It is convenient to regard $X_T$ as a process with values in the space of tempered distributions $\SP$ and prove convergence in this space. In the paper we prove a functional central limit theorem for the superprocess in Theorem \ref{thm:clt2} (the result for the BPS is already known  \cite[Theorem 2.1]{Mios:2009oq}). The theorem is in a sense classical as $F_T = T^{1/2}$ and the limit is Gaussian, namely a Wiener process.  The temporal structure of the limit is simple - the increments are independent, which contrasts sharply with the spatial structure being an $\SP$-valued Gaussian random field with the law depending on the properties of the Markov family $\eta$. This result can be explained by the subcriticality of the branching law. Below we present a shortened version 
of a heuristic argument presented in \cite{Mios:2009oq}. It uses a particle picture so it refers directly only to the BPS, nevertheless by the approximation presented in Section \ref{sec:approx} it is also applicable to the superprocess. The life-span of a family descending from one particle is short (its tail decays exponentially). Therefore, a particle hardly ever visits the same site multiple times. If we consider two distant, disjoint time intervals, it is likely that distinct (independent) families contribute to the increases of the occupation time in them. This results in independent increments of the limit process. Consequently, under mild assumptions, the properties of the movement play a minor role in the temporal part of the limit. On the other hand, the life-span of a family is too short to ``smooth out the grains in the space'' which, in turn, gives rise to the complicated spatial structure. 
	\item[Large and moderate deviation principles (LDP/MDP)] They are standard ways of studying rare events (on an exponential scale) when a random object converges to a deterministic limit. Moderate deviations can be also regarded as a link between the central limit theorem and large deviations (see Remark \ref{rem:moderate-devs} ). \\ In the paper we prove version of large deviation principles for the rescaled occupation process $Y_T$ for the BPS and the superprocess. They are contained in Theorem \ref{thm:ldp2} and Theorem \ref{thm:ldp}. The rate functions in these cases are quite complicated. Roughly speaking they are the Legendre transforms of functions expressed in terms of equations related to the systems. In both cases the results are not complete as the upper bounds and lower bounds are potentially not optimal and the upper bounds are derived only for a subclass of the compact sets. The reasons for this are to some extent fundamental. Exponential tightness is not likely to hold in this case, therefore a ``strong'' large deviation principle is impossible - see Remark \ref{rem:exp-tight}. Moreover, the functional approach is technically demanding. In Theorem \ref{thm:ldp2-one-dim} and Theorem \ref{thm:ldp-one-dim} we also present less powerful versions for the one-dimensional distribution which can be formulated more elegantly.  \\
The above limitations are not relevant to moderate deviation principles presented in Theorem \ref{thm:ldp22} and Theorem \ref{thm:ldp23}. We were able to obtain so-called strong deviation principles in a functional setting (i.e. for random variables taking values in the space $\mathcal{C}([0,1], \R)$). The rate functions in both cases are ``quite explicit''. What is more, the theorems closely resemble the Schilder theorem, which, together with the central limit theorems, strongly suggests that the large space-time scale behaviour is similar to the one of the Wiener process - see Remark \ref{rem:Schilder}. 
\end{description}
The distinctive feature of all results presented in the paper is the fact that they were obtained for a large class of Markov processes $\eta$. This is uncommon for stochastic models of this kind; usually $\eta$ is a well-known process (e.g. Brownian motion, $\alpha$-stable process, L\'evy process), which makes the analysis more tractable and explicit. The subcriticality of branching law suppresses the influence of the properties of $\eta$ (as it was discussed for the CLT), which makes it possible to carry out the reasoning in our fairly general setting.\\
We investigated the speed of convergence in the law of large numbers for the occupation time process $Y_T$ using two complementary tools: central limit theorems and large deviations, which together provide the full picture on various scales. The central limit theorems and moderate deviation principles indicate very close relation to the Brownian motion on large time scales which is slightly undermined by the large deviation principles. This phenomenon stems from the fact that in the ``exponential  scale'' of  the large deviation principles the properties of the Markov family $\eta$ finally play a role. While the results on the moderate deviations, given in this paper, seem quite satisfactory, it remains unclear whether the large deviation principles could be refined or are already ultimate. We stress that the results were obtained in the functional setting. The paper is written as a self-contained reference hence we summarise the results obtained earlier and present one-dimensional versions.\\
 We will now present our results against the state-of-the-art in the field. Central limit theorems for similar models with \emph{critical branching} were studied intensively by Bojdecki et al. and Milos. We just mention \cite{Bojdecki:2007ad,Milos:aa}, in which the reader finds further references. These systems (and related ones) were also studied using large deviation principles; \cite{Lee:1995aa,Deuschel:1994aa,Schied:1997bj, Li:2008bs} with \cite{Li:2008bs} describing the most recent developments. We also refer to \cite{Birkner:2007aa} as an example of similar results for branching random walks. The results for critical branching systems are qualitatively different from the ones presented here. The dependence on the properties of the particle movement is much stronger as the notion of transience and recurrence (for the movement itself and families of particles) plays vital role to the form of the limit - see also Remark \ref{rem:comparison-with}.\\
Systems with subcritical branching were largely neglected until recent works \cite{Mios:2009oq,Hong:2005aa}. While the studies of critical branching models concentrates mostly on the systems with ``well-behaving'' processes governing the particles movement (usually Brownian motion or $\alpha$-stable processes), \cite{Mios:2009oq,Hong:2005aa} admit a large class of Markov processes. This paper extends and virtually completes their developments. Firstly, we converted the functional central limit theorem for branching particle systems, \cite[Theorem 2.1]{Mios:2009oq}, to superprocesses, Theorem \ref{thm:clt2}. Secondly, we extend the results of \cite{Hong:2005aa} where the authors showed a large and moderate deviation principle for one-dimensional distribution of the superprocess \cite[Theorem 4.1 Theorem 5.1]{Hong:2005aa}. In the paper we present their functional counterparts and establish analogous results for the BPS. As it was mentioned above the only open issue left is the possibility of refining the large deviation principles.\\
Not surprisingly the proof techniques bear resemblance to the ones in \cite{Mios:2009oq,Hong:2005aa}. However they had to be enhanced to handle new situations. Loosely speaking, the main technical difficulty was to combine the methods of \cite{Mios:2009oq} suitable for functional setting with the methods of \cite{Hong:2005aa} developed to deal with large and moderate deviation principles. This required some delicate estimations of solutions of partial differential equations. The proof of the exponential tightness, which was the technically most cumbersome part, required also estimations of the suprema of stochastic processes. The BPS and the superprocesses are similar models and the proofs in both cases are similar, though usually more difficult for BPS.
%We developed a technique which makes it possible to convert proofs for superprocesses to proof for BPSs and vice versa (this direction is usually easy and works in many other situations) in many cases semi-automatically - see Lemma \ref{lem:comparison}. Though it is very simple we believe it may be of use for other researchers too.\\

The paper is organised as follows. In the next section we present the notation used throughout the paper. Section 3 is devoted to the detailed description of the superprocess. In Section 4 the results are presented. Finally, Section 5 contains the proofs.

\section{Notation}
In the whole paper we will use superscripts $^B$, $^S$ to indicate the BPS and the superprocesses, respectively. We shall skip the superscripts when a quantity (equation) will apply to both models or when it is clear from the context which model we are dealing with. By $\mathcal{B}(E)$ we will denote the Borel sets on space $E$. By $BV(E)$ we will denote the set of Borel measures on $E$ of bounded total variation.\\
 $\SP$ is a space of tempered distributions i.e. a nuclear space dual to the Schwartz space of rapidly decreasing functions $\SD$. The duality will be denoted by $\ddp{\cdot}{\cdot}$. By $\SP_+\subset \SP$ we will denote the subspace of positive functions.\\
In the whole paper 
\begin{equation}
	Q := V(1-2q),\label{def:Q} 
\end{equation}
which intuitively denotes the ``intensity of dying''. Recall that $V$ is the intensity of branching and $2q$ is the expected number of particles spawning from one particle. Clearly, the subcriticality of the branching law implies $Q>0$.\\
By $(\T{t})_{t\geq0}$ and $A$ we will denote, respectively, the semigroup and the infinitesimal operator corresponding to the Markov family $(\eta_t,\mathbb{P}_x)_{t\geq0,x\in \Rd}$ presented in Introduction. Sometimes instead of writing $\ev{}_x f(\eta_t)$ we write $\ev{f(\eta^x_t)}$.\\
For brevity of notation we also denote the semigroup 
\begin{equation*}
	\T{t}^Q f(x) := e^{-Qt}\T{t} f(x),
\end{equation*}
and the potential operator corresponding to it 
\begin{equation}
	\mathcal{U}^Q f(x) := \intc{+\infty} \T{t}^Q f(x) \dd{t}. \label{eq:potentialOperator}
\end{equation}
Three kinds of convergence are used. The convergence of finite-dimensional distributions is denoted by $\rightarrow_{fdd}$. For a continuous, $\SP$-valued process $X=(X_t )_{t\geq0}$ and any $\tau > 0$ one can define an $\mathcal{S}'(\mathbb{R}^{d+1})$-valued random variable 
\begin{equation}
	\ddp{\tilde{X}^\tau}{\Phi} := \intc{\tau}\ddp{X_t}{\Phi(\cdot,t)}\dd{t}. \label{eq:space-time-method} 
\end{equation}
If for any $\tau > 0$ $\tilde{X}_n \rightarrow \tilde{X}$ in distribution, we say that the convergence in the space-time sense holds and denote this fact by $\rightarrow_i$. Finally, we consider the functional weak convergence denoted by $X_n\rightarrow_c X$. It holds if for any $\tau > 0$ processes $X_n = (X_n (t))_{t\in[0,\tau]}$ converge to $X = (X(t))_{t\in[0,\tau ]}$ weakly in $\mathcal{C}([0, \tau ], \SP)$ (in the sequel without loss of generality we assume $\tau = 1$ and skip the superscript). It is known that $\rightarrow_i$ and $\rightarrow_{fdd}$ do not imply each other, but either of them together with tightness implies $\rightarrow_c$. Conversely, $\rightarrow_c$ implies both $\rightarrow_i$ , $\rightarrow_{fdd}$.\\
For a measure $\nu\in BV([0,1])$ we write
\begin{equation}
	\chi_\nu(s) := \nu((s,1]), \quad \chi_{\nu,T}(t) := \chi_\nu (t/T). \label{def:measure-chi} 
\end{equation}
We will skip the subscript $\nu$ when the measure is obvious from the context. By $H^1\subset \mathcal{C}([0,1],\mathbb{R})$ we denote the space of functions $f$ for which there exists $f'$ such that $f(t) = \intc{t}f'(s)\dd{s}$ and $f'$  is square integrable. We denote
\[
	\norm{f}{H^1}{} := \norm{f'}{L^2}{}.
\]
We also define
\begin{equation}
	\mathcal{C}_{a,b} := \cbr{f \in \mathcal{C}([0,1], \mathbb{R}): f \text{ is differentiable and }  \forall_{x\in[0,1]} f'(x) \in (a,b)}. \label{eq:space-c} 
\end{equation}
For a function $f\in \mathcal{C}([0,1], \R)$ we denote its modulus of continuity
\begin{equation}
	w(f,\delta) := \sup_{\substack{s,t \in [0,1]\\|s-t|<\delta}} |f(s)-f(t)|. \label{eq:modulus} 
\end{equation}
By $c,c_1,\ldots,C,C_1,\ldots$ we will denote generic constants.

\section{Subcritical superprocess with immigration} \label{sec:approx} In this section we recall the construction of the superprocess. By $(N^n_t)_{t\geq0}$ we denote the $n$-th approximation of superprocess i.e. the branching particle system in which particles live for exponential time with parameter $V_n = 2 n V q$. The branching law is given by  a generating function 
\begin{equation}
	F_n(s) = q_n s^2 + (1-q_n), \quad q_n:= \frac{ 2nq + 2q-1}{4nq}. \label{eq:generating2} 
\end{equation}
The system starts from the null measure (the starting measure does not affect the results, hence this assumption can be easily dropped) and the immigration is given by a space-time homogenous Poisson random field with intensity $nH (\lambda_{d+1})$. We also assume that each of the particles carries mass $1/n$. The particular choice of $V_n,q_n$ is to some extent arbitral and was made to keep the intensity of dying fixed at level $Q$ and for the sake of convenience (e.g. to have the same constants in forthcoming equations \eqref{eq:v-integral}, \eqref{eq:w-integral-super}). It does not affect the generality of the results as one can easily reformulate theorems for any other choice.\\
By $N$ we denote a measure-valued homogenous Markov process with the following Laplace transform 
\begin{equation*}
	\ev \exp\rbr{-\ddp{N_t}{\varphi}} = \exp\cbr{-H\intc{t} \ddp{\lambda}{H_s \varphi} \dd{s}},
\end{equation*}
where $\varphi:\Rd\mapsto \R_+$ is measurable and $H_s$ is a semigroup given by equation 
\begin{equation*}
	H_t \varphi(x) = \T{t}^Q \varphi(x) - Vq \intc{t} \T{t-s}^Q \sbr{(H_s \varphi)^2(\cdot)}(x) \dd{s}.
\end{equation*}
The process $N$ will be called the { superprocess related  to the BPS}, which is justified by
\begin{prop} \label{prop:superprocess-convergence} 
	The following convergence holds 
	\begin{equation*}
		N^n \rightarrow_c N.
	\end{equation*}
\end{prop}
The proof is standard. For instance, one can follow the lines of \cite[Section 1.4]{Etheridge:2000fe}). 

\section{Results} \label{sec:res}
Firstly we present the restrictions imposed on the Markov family $(\eta_t,\mathbb{P}_x)_{t\geq0,x\in \Rd}$. They are mild and easy to check in concrete cases. Let us denote the quadratic forms
\begin{equation}
	T_1(\varphi) := \norm{\mathcal{U}^Q\rbr{\varphi \: \mathcal{U}^Q\varphi}}{1}{},\quad \varphi \in \SD, \label{eq:T1} 
\end{equation}
\begin{equation}
	T_2(\varphi) := \norm{\mathcal{U}^Q \rbr{ (\mathcal{U}^Q\varphi)^2 }}{1}{}, \quad \varphi \in \SD, \label{eq:T2}  
\end{equation}
also, slightly abusing notation, we will also use $T_1$ and $T_2$ to denote the corresponding {bilinear forms}.
\subsection{Assumptions}
\begin{enumerate}
	\item[(A1)] The Markov family $(\eta_t,\mathbb{P}_x)_{t\geq0,x\in \Rd}$ is {almost uniformly stochastically continuous} i.e. 
	\begin{equation*}
		\forall_{r>0}\forall_{\epsilon>0} \liminf_{s\rightarrow 0} \inf_{ |x|\leq r } \mathbb{P}_x(\eta_s,B(x,\epsilon)) = 1, 
	\end{equation*}
	where $B(x,\epsilon)$ denotes the ball of radius $\epsilon$ with the centre in $x$. %Additionally, we assume that for any $x$ trajectories of process are almost surely bounded on any finite interval. 
	\item[(A2)] Let $D_A$ denotes the domain of the infinitesimal operator $A$. We have 
	\begin{equation*}
		\SD \subset D_A. 
	\end{equation*}
	\item[(A3)] For any $\varphi\in \SD$  the semigroup $(\T{t}^\varphi)_{t\geq 0}$ given by 
	\begin{equation*}
		\T{t}^\varphi f(x) := \ev{}_x \exp\cbr{\intc{t} \varphi(\eta_s) \dd{s}} f(\eta_t), 
	\end{equation*}
	is a Feller semigroup. 
	\item[(A4)] For any $\varphi\in \SD$ 
	\begin{equation}
		T_1(\varphi) <+\infty, \quad T_2(\varphi) <+\infty. \label{ass:1} 
	\end{equation}
	\item[(A5)] For any $\varphi\in \SD$ 
	\begin{equation}
		t^{3/2}  \norm{\T{t}^Q \varphi}{1}{} \rightarrow 0. \label{ass:3} 
	\end{equation}
	\item[(A6)] For any $h\in \mathcal{L}^2$ there exist $C>0$ and $Q'>0$ such that
	\begin{equation*}
		\norm{\T{t}^Q h}{2}{} \leq  C e^{-Q't} \norm{h}{2}{}, \: \forall_{t\geq 0}. \label{}
	\end{equation*}
	\item[(A7)] For any $h\in \mathcal{L}^1$ there exist $C>0$ and $Q'>0$ such that
	\begin{equation*}
		\norm{\T{t}^Q h}{1}{} \leq C e^{-Q't}\norm{h}{1}{}, \: \forall_{t\geq 0}. \label{}
	\end{equation*}
	\item[(A8)] For any $\varphi \in \SD_+$ there exist $\epsilon>0,c>0$ such that  
	\begin{equation*}
		\norm{\T{t}^Q \varphi}{1}{} \leq c \rbr{1\wedge t^{-2-\epsilon}}. 
	\end{equation*}
	\item[(A9)] For any $\varphi \in \SD_+$ there exist $\epsilon>0,c>0$ such that  and for all $h,l$ 
	\begin{equation*}
		\norm{\T{t}^Q\sbr{\T{h}^Q\varphi(\cdot) \T{l}^Q\varphi(\cdot) }}{1}{} \leq c\rbr{1\wedge t^{-2-\epsilon}}. 
	\end{equation*}
\end{enumerate}

% \begin{rem}
% 
% 	The assumption (A1) can be replaced by a stronger, but more natural, condition as follows. The Markov family $(\eta_t,\mathbb{P}_x)_{t\geq0,x\in \Rd}$ is \emph{uniformly stochastically continuous} i.e. 
% 	\begin{equation*}
% 		\sup_x \mathbb{P}_x(\eta_s,B(x,\epsilon)) \rightarrow 1, \quad \text{ as } s\rightarrow 0, 
% 	\end{equation*}
% \end{rem}
\begin{rem}
	The assumptions above are used in various configurations and are not independent. E.g. (A7) implies (A5), (A8) and (A9).
\end{rem}

\begin{rem}
	The conditions above can be easily checked for concrete processes. For example they hold for any L\'evy process. Consider also the Ornstein-Uhlembeck process $\cbr{\eta_t}_{t\geq 0}$ given by the stochastic equation
	\[
		d\eta_t = - \theta \eta_t \dd{t} + \sigma dW_t,\,
	\]
	where $\theta>0,\sigma >0$ and $W$ is the Wiener process. $X$ fulfils the above assumptions if 
	\[
		\theta < Q.
	\]
	This condition has a clear interpretation. $\theta$ determines the speed at which particles arrive in proximity of  $0$ (this interpretation would be totally strict if $\sigma=0$). The intensity of dying i.e. $Q$ have to be large enough to prevent clumping particles near $0$.
\end{rem}

\subsection{Branching process} In this subsection $N$ denotes the BPS described in Introduction. The processes \eqref{eq:rescaled-occupation} and \eqref{def:occupation-process} are defined with this $N$. Firstly we recall the central limit theorem \cite[Theorem 2.1]{Mios:2009oq} 
\begin{thm}
	\label{thm:clt1} Let $X_T$ be the rescaled occupation time fluctuations process given by (\ref{def:occupation-process}). Assume that $F_T = T^{1/2}$ and assumptions (A1)-(A5) are fulfilled. Then 
	\begin{equation*}
		X_T \rightarrow_{i} X,\quad \text{ and }\quad X_T \rightarrow_{fdd} X, 
	\end{equation*}
	where $X$ is a generalised $\SP$-valued Wiener process with covariance functional 
	\begin{equation*}
		Cov(\ddp{X_t}{\varphi_1}, \ddp{X_s}{\varphi_2}) = H\rbr{s\wedge t} \rbr{T_1(\varphi_1, \varphi_2) + Vq T_2(\varphi_1, \varphi_2)}, \quad \varphi_1,\varphi_2 \in \SD, 
	\end{equation*}
	if, additionally, assumptions (A8)-(A9) are fulfilled then 
	\begin{equation*}
		X_T \rightarrow_{c} X. 
	\end{equation*}
\end{thm}
\begin{rem} \label{rem:comparison-with}
	The limit is an $\SP$-valued Wiener process with a simple time structure and a complicated temporal one (for any $d$). This  resembles the result for the system with critical branching in large dimensions (e.g. \cite{Bojdecki:2006aa}, \cite{Milos:2008aa}). The main reason of this is a short (exponentially-tailed) life-span of a family descending from one particle. It leads to independent increments in the limit (as there are no “related” particles in the long term). On the other hand the movement is “not strong enough” to smooth out the spatial structure.
\end{rem}
\noindent Let us now recall \eqref{def:Q} and denote
\begin{equation}
	Q_0^B = V(1-2\sqrt{q(1-q)}) . \label{eq:q0b}%Q+2Vq-2\sqrt{QVq+(Vq)^2}
\end{equation}
A LDP contained in Theorem \ref{thm:ldp2} is our next objective. Before that we need
\begin{lem}
	\label{lem:convergence-v2} Let $\varphi \in \SD_+$ and  $\theta < Q_0^B/\norm{\varphi}{\infty}{}$. Then the equation
	\begin{equation*}
		v_{\varphi}(x,t,\theta) = \intc{t} \T{t-s}^Q \theta\varphi(x) + \intc{t} \T{t-s}^Q \rbr{\theta\varphi(\cdot,s) v_{\varphi}(\cdot,s,\theta)+ Vq v_{\varphi}^2(\cdot,s,\theta)}(x) \dd{s},
	\end{equation*}
	has a unique solution and the limit below is finite
	\begin{equation}
		v_\varphi(x, \theta) := \lim_{t\ti} v_{\varphi} (x,t,\theta). \label{eq:vvv2} 
	\end{equation}
\end{lem}
\noindent The proof is deferred to Section \ref{sec:one-particle}. We define now 
\begin{equation}
	A :=  \lim_{\theta \rightarrow Q_0^B/\norm{\varphi}{\infty}{}}\frac{\partial}{\partial \theta} v_\varphi(x,\theta).  \label{eq:def-A1}
\end{equation}
Fix $\varphi \in \SD_+$ and let us recall \eqref{def:measure-chi}. For $\nu\in BV([0,1])$ such that $\norm{\chi_\nu}{\infty}{}< Q_0^B/\norm{\varphi}{\infty}{}$ we define
\[
	\Lambda_\varphi(\nu) := \intc{1} \intr v_\varphi(x,\chi_\nu(t)) \dd{t} \dd{x},
\]
\[
	\Lambda_\varphi^*(f) := \sup_{\nu \in B }\sbr{ \ddp{f}{\nu} - \Lambda_\varphi(\nu)},
\]
where $B = \cbr{\nu\in BV([0,1]):\norm{\chi_\nu}{\infty}{} <Q_0^B/\norm{\varphi}{\infty}{} }$. This closely resembles the Legendre transform. Now we can formulate a large deviation principle

% \begin{lem}
% 	\label{lem:xyz12} For any $\chi \in C([0,1],\R)$ there exist $\overline{\theta}>0$ such that 
% 	\begin{equation*}
% 		\theta \mapsto\intc{1} \intr u_\varphi(x,\theta \chi(t)) \dd{t} \dd{x} 
% 	\end{equation*}
% 	is well-defined and analytic on certain open $U \subset \mathbb{C}, 0\in U$.% such that $(-\overline{\theta} , \overline{\theta})\times i 0 \subset U$. 
% \end{lem}
% 
% Recall (\ref{def:measure-chi}), given $\nu$ define 
% \begin{equation}
% 	u_\varphi(\theta \nu)= \intc{1} \intr u_\varphi(x, \theta\chi_\nu(t)) \dd{t} \dd{x} \label{eq:try2} 
% \end{equation}
% By Lemma \ref{lem:xyz12} it is analytic (with respect to $\theta$) and well-defined on certain open set of $\mathbb{C}$. Consider now maximal analytic extension and put $ +\infty $ where function is not defined. Slightly abusing notation we will denote this function also with $ u_\varphi $. 
\begin{thm} 
	\label{thm:ldp2} Let $\varphi \in \SD_+$ and  $Y_T$ be the rescaled occupation time process given by \eqref{eq:rescaled-occupation}. Assume that $F_T = T$ and assumptions (A1)-(A3) are fulfilled. Then for any open set  $U \subset \mathcal{C}([0,1],\R)$,
	\begin{equation*}
		\liminf_{T\rightarrow +\infty} T^{-1} \log \pr{ \ddp{Y_T}{\varphi} \in U} \geq -\inf_{f\in U \cap \mathcal{C}_{0,A/H}} \Lambda^*(f). 
	\end{equation*}
	For any $f\in \mathcal{C}_{0,A/H}$ and any $\delta>0$ there exists $r>0$ such that
	\begin{equation*}
		\limsup_{T\rightarrow +\infty} T^{-1} \log \pr{ \ddp{Y_T}{\varphi} \in B(f,r)} \leq \delta - \Lambda^*(f),
	\end{equation*}
	where $B(f,r)$ is a ball in $\mathcal{C}([0,1], \R)$ of radius $r$ centred at $f$.
\end{thm}
\begin{rem} \label{rem:exp-tight}
	We checked that for certain Markov families $\eta$ exponential tightness does not hold. Consequently, in these cases a strong large deviation principle cannot hold either. We conjecture that this phenomenon is general and exponential tightness does not hold for any Markov family $\eta$.
\end{rem}
\begin{rem}
	In the lower-bound formula the restriction to $\mathcal{C}_{0,A/H}$ is fairly acceptable. The paths of $Y_T$ are continuous and non-decreasing. Hence large class of open sets $U$ can be ``well-approximated'' by $\mathcal{C}_{0,A/H}$ in a sense that any function in $U \setminus \mathcal{C}_{0,A/H}$ has to increase ``very fast'' on some intervals. This requires a lot particles to gather in a small set which is not very likely in our system.
\end{rem}
\begin{rem}
	The restriction in the lower-bound case is more awkward. We conjecture that  
	\begin{equation*}
		\limsup_{T\rightarrow +\infty} T^{-1} \log \pr{ \ddp{Y_T}{\varphi} \in K} \leq -\inf_{f\in K \cap \mathcal{C}_{0,A/H}} \Lambda^*(f). 
	\end{equation*}
		is true for some class of compact sets $K$.
\end{rem}
To give full picture we also recall here a non-functional counterpart of the above theorem. Since Theorem \ref{thm:ldp2} is a weak version of large deviations we cannot use the contraction principle \cite[Theorem 4.2.1]{Dembo:1998fu} and the theorem below requires a separate proof (which obviously is much simpler than the one of Theorem \ref{thm:ldp2} and hence skipped). 
\begin{thm} \label{thm:ldp2-one-dim} 
	 Let $\varphi \in \SD_+$ and  $Y_T$ be the rescaled occupation time process given by \eqref{eq:rescaled-occupation}. Assume that $F_T = T$ and assumptions (A1)-(A3) are fulfilled. Then there exists $\delta>0$ such that for any open set $U \subset (0,\delta)$, closed set $L\subset (0,\delta)$ 
	\begin{equation}
		\liminf_{T\rightarrow +\infty} T^{-1} \log \pr{ \ddp{Y_T(1)}{\varphi} \in U} \geq -\inf_{x\in U} \Lambda^*(x), \label{eq:mdp-upper} 
	\end{equation}
	\begin{equation*}
		\limsup_{T\rightarrow +\infty} T^{-1} \log \pr{ \ddp{Y_T(1)}{\varphi} \in L} \leq -\inf_{x\in L} \Lambda^*(x),
	\end{equation*}
	where 
	\begin{equation*}
		\Lambda^*(x) = \sup_{\theta \leq Q^B_0/\norm{\varphi}{\infty}{}} \sbr{ x\theta - \ddp{v_\varphi(\cdot, \theta )}{\lambda} }, 
	\end{equation*}
	where $v_\varphi$ is given by \eqref{eq:vvv2}. 
\end{thm}
Now we present a strong moderate deviation principle.
\begin{thm}
	\label{thm:ldp22} Let $\varphi\in \SD_+ $ and $X_T$ be the rescaled occupation time fluctuations process given by (\ref{def:occupation-process}). Assume that $0<\alpha<1$, $F_T = T^{(1+\alpha)/2}$  and assumptions (A1)-(A7) are fulfilled. Then, for any open set $U\subset \mathcal{C}([0,1], \mathbb{R})$ and any closed set $L\subset \mathcal{C}([0,1], \mathbb{R})$ we have 
	\begin{equation*}
		\liminf_{T\rightarrow +\infty} T^{-\alpha} \log \pr{\ddp{{X}_T}{\varphi} \in U} \geq -\inf_{f\in U} \Lambda^*(f),
	\end{equation*}
	\begin{equation*}
		\limsup_{T\rightarrow +\infty} T^{-\alpha} \log \pr{\ddp{{X}_T}{\varphi} \in L} \leq -\inf_{f\in L} \Lambda^*(f),
	\end{equation*}
	where 
	\begin{equation*}
		\Lambda^*(f) = \frac{\norm{f}{H^1}{2}}{4H\rbr{T_1(\varphi) + Vq T_2(\varphi)} }, %\label{eq:legandr} 
	\end{equation*}
	if $f\in H^1$ and $\Lambda^*(f) = \infty$ if $f\notin H^1$.
\end{thm}
% \begin{rem}
% 	The bounds in this theorem are optimal and can be calculated explicitly. It is also worthwhile to mention that the exponential tightness holds in this case hence we were able to obtain a strong deviation principle.
% \end{rem}
\begin{rem} \label{rem:moderate-devs} 
	This theorem ``links'' the central limit theorem and the large deviation principle. Roughly speaking $\alpha\rightarrow 0$ corresponds to Theorem \ref{thm:clt1} and $\alpha\rightarrow 1$ to Theorem \ref{thm:ldp2}.
\end{rem}
\begin{rem} \label{rem:Schilder} 
	Let us also notice a close resemblance of this result to the Schilder theorem \cite[Theorem 5.2.3]{Dembo:1998fu} which is a strong large deviation principle for the Wiener process. This together with Theorem \ref{thm:clt1} imply that the process of fluctuations of the occupation time is much alike the Wiener process. While the CLT establishes this fact for ``typical'' paths the MDP complements it to ``moderately rare'' events. That also means that the properties of the movement of the particles merely influence the properties of converge. It should be noted however that in the LDP of Theorem \ref{thm:ldp2} the analogy breaks, meaning that for ``extremely rare'' events the properties of the movement finally commence to play a significant role. 
\end{rem}
We present also a moderate deviation principle for $X_T(1)$.
\begin{corollary}\label{cor:a}
	Let $\varphi\in \SD_+$ and $X_T$ be the rescaled occupation time fluctuations process given by (\ref{def:occupation-process}). Assume that $0<\alpha<1$, $F_T = T^{(1+\alpha)/2}$  and assumptions (A1)-(A7) are fulfilled. Then for any open set $U\subset \R$ and any closed set $L\subset \R$ we have 
	\begin{equation*}
		\liminf_{T\rightarrow +\infty} T^{-\alpha} \log \pr{\ddp{{X}_T(1)}{\varphi} \in U} \geq -\inf_{x\in U} \Lambda^*(x),
	\end{equation*}
	\begin{equation*}
		\limsup_{T\rightarrow +\infty} T^{-\alpha} \log \pr{\ddp{{X}_T(1)}{\varphi} \in L} \leq -\inf_{x\in L} \Lambda^*(x),
	\end{equation*}
	where 
	\begin{equation*}
		\Lambda^*(x) = \frac{x^2}{4H\rbr{T_1(\varphi) + Vq T_2(\varphi)} }.
	\end{equation*}
\end{corollary}
The proof is an easy application of the contraction principle \cite[Theorem 4.2.1]{Dembo:1998fu} to Theorem \ref{thm:ldp22}.

\subsection{Superprocess} \label{sec:subcritical} In this subsection $N$ denotes the superprocess described in Section \ref{sec:approx}. The processes \eqref{eq:rescaled-occupation} and \eqref{def:occupation-process} are defined with this $N$. Firstly we present a central limit theorem 
\begin{thm}
	\label{thm:clt2} Let $X_T$ be the rescaled occupation time fluctuations process given by (\ref{def:occupation-process}). Assume that $F_T = T^{1/2}$ and assumptions (A1)-(A5) are fulfilled. Then 
	\begin{equation*}
		X_T \rightarrow_{i} X,\quad \text{ and }\quad X_T \rightarrow_{fdd} X, 
	\end{equation*}
	where $X$ is a generalised $\SP$-valued Wiener process with covariance functional 
	\begin{equation*}
		Cov(\ddp{X_t}{\varphi_1}, \ddp{X_s}{\varphi_2}) = VHq \rbr{s\wedge t} T_2(\varphi_1, \varphi_2)\quad \varphi_1,\varphi_2 \in \SD,  
	\end{equation*}
	if, additionally, assumptions (A8)-(A9) are fulfilled then 
	\begin{equation*}
		X_T \rightarrow_{c} X. 
	\end{equation*}
\end{thm}
\noindent Let us recall \eqref{def:Q} and denote
\begin{equation}
	Q^S_0 := \frac{Q^2}{4Vq}. \label{eq:q0S} 	
\end{equation}
The LDP contained in Theorem \ref{thm:ldp} is our next aim. To this end we need to formulate the following lemma
\begin{lem}
	\label{lem:convergence-v} Let $\varphi \in \SD_+$ and $\theta< Q^S_0/ \norm{\varphi}{\infty}{} $. Then the equation 
	\begin{equation*}
		v_{\varphi}(x,t,\theta) = \intc{t} \T{t-s}^Q  \theta \varphi(x) + Vq \intc{t} \T{t-s}^Q v_{\varphi}^2(\cdot,s,\theta)(x) \dd{s},
	\end{equation*}
	has a unique solution and the limit below is finite
	\begin{equation}
		v_\varphi(x, \theta) := \lim_{t\ti} v_{\theta\varphi} (t,x, \theta). \label{eq:vvv} 
	\end{equation}
\end{lem}
The proof follows the lines of the proof of Lemma \ref{lem:convergence-v2} and is skipped. We define now 
\begin{equation}
	A :=  \lim_{\theta \rightarrow Q^S_0/ \norm{\varphi}{\infty}{}}  \frac{\partial}{\partial \theta} v_\varphi(x,\theta). \label{eq:def-A2}
\end{equation}
Fix $\varphi \in \SD_+$ and let us recall \eqref{def:measure-chi}. For $\nu$ such that $\norm{\chi_\nu}{\infty}{}< Q^S_0/\norm{\varphi}{\infty}{}$ we define
\begin{equation}
	\Lambda_\varphi(\nu) = \intc{1} \intr v_\varphi(x,\chi_\nu(t)) \dd{t} \dd{x}, \label{eq:legendre-lambda} 
\end{equation}
\begin{equation}
	\Lambda_\varphi^*(f) = \sup_{\nu \in B }\sbr{ \ddp{f}{\nu} - \Lambda_\varphi(\nu)}, \label{eq:crippled-legendre} 
\end{equation}
where $B = \cbr{\nu:\norm{\chi_\nu}{\infty}{} <Q^S_0/\norm{\varphi}{\infty}{} }$. This closely resembles the Legendre transform. Now we can formulate a large deviation principle 
\begin{thm}
	\label{thm:ldp} Let $\varphi \in \SD_+$ and  $Y_T$ be the rescaled occupation time process given by \eqref{eq:rescaled-occupation}. Assume that $F_T = T$ and assumptions (A1)-(A3) are fulfilled. Then for any open set $U \subset \mathcal{C}([0,1],\R)$,
		\begin{equation*}
			\liminf_{T\rightarrow +\infty} T^{-1} \log \pr{ \ddp{Y_T}{\varphi} \in U} \geq -\inf_{f\in U \cap \mathcal{C}_{0,A/H}} \Lambda^*(f).
		\end{equation*}
		For any $f\in \mathcal{C}_{0,A/H}$ and any $\delta>0$ there exists $r>0$ such that
		\begin{equation*}
			\limsup_{T\rightarrow +\infty} T^{-1} \log \pr{ \ddp{Y_T}{\varphi} \in B(f,r)} \leq \delta - \Lambda^*(f). 
		\end{equation*}
		where $B(f,r)$ is a ball in $\mathcal{C}([0,1],\R)$ of radius $r$ centred at $f$.
\end{thm}
To give full picture we also recall here a non-functional counterpart of the above theorem. It is a slightly modified version of \cite[Theorem 4.1]{Hong:2005aa}.
\begin{thm} \label{thm:ldp-one-dim}
	 Let  $\varphi\in \SD_+$ and $Y_T$ be the rescaled occupation time process given by \eqref{eq:rescaled-occupation}. Assume that $F_T = T$ and assumptions (A1)-(A3) are fulfilled. Then there exists $\delta>0$ for any open set $U \subset (0,\delta)$ and any closed set $L\subset (0,\delta)$ we have 
	\begin{equation*}
		\liminf_{T\rightarrow +\infty} T^{-1} \log \pr{ \ddp{Y_T(1)}{\varphi} \in U} \geq -\inf_{x\in U} \Lambda^*(x),
	\end{equation*}
	\begin{equation*}
		\limsup_{T\rightarrow +\infty} T^{-1} \log \pr{ \ddp{Y_T(1)}{\varphi} \in L} \leq -\inf_{x\in L} \Lambda^*(x),
	\end{equation*}
	where 
	\begin{equation*}
		\Lambda^*(x) = \sup_{\theta \leq Q^S_0/ \norm{\varphi}{\infty}{}} \sbr{ x\theta - \ddp{v_\varphi(\cdot, \theta )}{\lambda} }, 
	\end{equation*}
	where $v_\varphi(\theta)$ is given by \eqref{eq:vvv}. 
\end{thm}
Now we present a strong moderate deviation principle
\begin{thm}
	\label{thm:ldp23} Let $\varphi \in \SD_+$ and $X_T$ be the rescaled occupation time fluctuations process given by (\ref{def:occupation-process}). Assume that $F_T = T^{(1+\alpha)/2}$, $0<\alpha<1$ and assumptions (A1)-(A7) are fulfilled. Then, for any open set $U \subset \mathcal{C}([0,1], \mathbb{R})$ and any closed set $L\subset \mathcal{C}([0,1], \mathbb{R})$ we have 
	\begin{equation*}
		\liminf_{T\rightarrow +\infty} T^{-\alpha} \log \pr{\ddp{{X}_T}{\varphi} \in U} \geq -\inf_{f\in U} \Lambda^*(f),
	\end{equation*}
	\begin{equation*}
		\limsup_{T\rightarrow +\infty} T^{-\alpha} \log \pr{\ddp{{X}_T}{\varphi} \in L} \leq -\inf_{f\in L} \Lambda^*(f),
	\end{equation*}
	where 
	\begin{equation*}
		\Lambda^*(f) = \frac{\norm{f}{H^1}{2}}{4H Vq T_2(\varphi) }, \label{eq:legare2} 
	\end{equation*}
	if $f\in H^1$ and $\Lambda^*(f) = \infty$ if $f\notin H^1$.
\end{thm}
The contraction principle \cite[Theorem 4.2.1]{Dembo:1998fu} can be applied here to obtain a moderate deviation principle for $X_T(1)$. This result was already known \cite[Theorem 5.1]{Hong:2005aa}, we put it here for readers' convenience 
\begin{corollary}
	Let $\varphi\in \SD_+$ and $X_T$ be the rescaled occupation time fluctuations process given by (\ref{def:occupation-process}). Assume that $0<\alpha<1$, $F_T = T^{(1+\alpha)/2}$  and assumptions (A1)-(A7) are fulfilled. Then, for any open set $U \subset \R$ and any closed set $L\subset \R$ we have 
	\begin{equation*}
		\liminf_{T\rightarrow +\infty} T^{-\alpha} \log \pr{\ddp{{X}_T(1)}{\varphi} \in U} \geq -\inf_{x\in U} \Lambda^*(x),
	\end{equation*}
	\begin{equation*}
		\limsup_{T\rightarrow +\infty} T^{-\alpha} \log \pr{\ddp{{X}_T(1)}{\varphi} \in L} \leq -\inf_{x\in L} \Lambda^*(x),
	\end{equation*}
	where 
	\begin{equation*}
		\Lambda^*(x) = \frac{x^2}{4H Vq T_2(\varphi) }.
	\end{equation*}
\end{corollary}
\begin{rem}
Remarks concerning the BPS contained in the previous subsection are also valid for the superprocess.
\end{rem}

\section{Proofs} \label{sec:proof} 
The proofs for the branching particle system and the superprocess are similar and are presented together.
\subsection{Notation} \label{sec:notation} 
 In the proofs we use the following notation. Firstly,  $\Phi$ is always of either of two forms 
\begin{equation}
	\Phi(x,s) = \varphi(x)\psi(s), \varphi\in \SD, \psi \in \mathcal{S}(\mathbb{R}). \label{def:simplification} 
\end{equation}
\begin{equation*}
	\Phi(x, \dd{s}) = \varphi(x) \nu(\dd{s} ), \varphi\in \SD, \nu \text{ is a signed measure of finite variation}
\end{equation*}
For the $\Phi$ of the first form we define
\begin{equation}
	\varphi_{T}\left(x\right):=\frac{1}{F_{T}}\varphi\left(x\right), \:\chi(s) := \int_s^1\psi(u) \dd{u},\: \chi_{T}:=\chi\left(\frac{t}{T}\right). \label{def:notation-T} 
\end{equation}
and for the $\Phi$ of the second form we use $\chi$ given by \eqref{def:measure-chi}. Note that it is a c\'adl\'ag function. Secondly, throughout the paper $\Psi$ always is
\begin{equation*}
	\Psi(x,s) = \varphi(x)\chi(s), 
\end{equation*}
\begin{equation}
	\Psi_{T}\left(x,s\right)=\frac{1}{F_{T}}\Psi\left(x,\frac{s}{T}\right) = \varphi_T(x)\chi_T(s). \label{def:simpl} 
\end{equation} 
In the following the notation is always assumed unless stated otherwise.

\subsection{One-particle equation} \label{sec:one-particle}
In this section we present an equation describing the behaviour of the occupation time for the branching system starting from a single particle. Subsequently we also acquire its analogue for the superprocess. These equations play key role in the rest of the proofs.
Recall  \eqref{eq:generating} and define $G(s) := F(1-s) - (1-s)$ 
\[
	G(s) = q s^2 + (1-2q) s.
\]
We denote the Laplace transform of the occupation time for the system starting off from a single particle at $x$
\begin{equation}
	v^B_{\Psi}\left(x,r,t\right):=\mathbb{E}\exp\left\{ \int_{0}^{t}\left\langle N_{s}^{x},\Psi\left(\cdot,r+s\right)\right\rangle \dd{s}\right\}-1, \:\: \Psi \in \mathcal{S}(\mathbb{R}^{d+1}) \label{def:v}, 
\end{equation}
where $\cbr{N^{x}_s}_{s\geq 0}$ denotes the empirical measure of the particle system with the initial condition $N_0^{x} = \delta_x$. $N^{x}$ is a system in which particles evolve according to the dynamics described in Introduction but without immigration.
\begin{lem}
	\label{lem:equation} Let $Q_0^B$ be given by \eqref{eq:q0b}. Assume that $\Psi < Q_0^B$ and assumptions (A1)-(A3) are fulfilled then 
	\begin{equation}
		v^B_{\Psi}\left(x,r,t\right)=\int_{0}^{t}\mathcal{T}^Q_{t-s}\left[\Psi\left(\cdot,r+t-s\right)(1+v^B_\Psi\left(\cdot,r+t-s,s\right)) + Vq  (v^B_{\Psi}\left(\cdot,r+t-s,s\right))^2  \right]\left(x\right)\dd{s}.\label{eq:v-integral} 
	\end{equation}
\end{lem}
The proof of this lemma is a rather obvious modification of the proof \cite[Lemma 3.1]{Mios:2009oq} though one needs to be careful as the equation may blow up for some $\Psi$. The proof of \cite[Lemma 3.2]{Hong:2005aa} can be modified to exclude this possibility and to show that $v^B$ admits a unique solution.\\
Now we move to the analogue for the superprocess. Let $v^S_{\Psi}$ be analogue of \eqref{def:v} for the superprocess i.e. this time by $N^x$ we understand the superprocess with the initial condition $\delta_x$ again without immigration. 
\begin{lem}
	\label{lem:equation-super} Let $Q_0^S$ be given by \eqref{eq:q0S}. Assume that  ${\Psi} < Q_0^S$ and assumptions (A1)-(A3) are fulfilled then 
	\begin{equation}	v^S_{\Psi}\left(x,r,t\right)=\int_{0}^{t}\mathcal{T}^Q_{t-s}\left[\Psi\left(\cdot,r+t-s\right)+Vq\left(v^S_{\Psi}\left(\cdot,r+t-s,s\right)\right)^2\right]\left(x\right)\dd{s}.\label{eq:w-integral-super} 
	\end{equation}
\end{lem}
\begin{rem} With additional assumption that $\Psi\geq 0$ we can define (to indicate the difference we use $\bar{\:}\:$)
	\begin{equation*}
		\bar{v}^S_{\Psi}\left(x,r,t\right)=1-\mathbb{E}\exp\left\{ -\int_{0}^{t}\left\langle N_{s}^{x},\Psi\left(\cdot,r+s\right)\right\rangle \dd{s}\right\}, \Psi \in \mathcal{S}(\mathbb{R}^{d+1}), \Psi\geq0, 
	\end{equation*}
	which fulfils the equation
	\begin{equation}	\bar{v}^S_{\Psi}\left(x,r,t\right)=\int_{0}^{t}\mathcal{T}^Q_{t-s}\left[\Psi\left(\cdot,r+t-s\right)-Vq\left(\bar{v}^S_{\Psi}\left(\cdot,r+t-s,s\right)\right)^2\right]\left(x\right)\dd{s}. \label{eq:integral-super-simple} 
	\end{equation}
	This equation is much much simpler in analysis but can be used only in the proof of the CLT. In the proof of the deviation principles we will need \eqref{def:v}.
\end{rem}

\begin{proof}
 Recall that $F_n$ denotes the generating function of the branching law \eqref{eq:generating2}. The one particle equation from Lemma \ref{lem:equation} for the $n$-th approximation is
	\begin{equation*}
		v^B_{n,\Psi}\left(x,r,t\right)=\int_{0}^{t}\mathcal{T}_{t-s}^{Q_n} \left[\Psi\left(\cdot,r+t-s\right)\left(1+v^B_{n,\Psi}\left(\cdot,r+t-s,s\right)\right)+V_n q_n \left(v^B_{n,\Psi}\left(\cdot,r+t-s,s\right)\right)^2\right]\left(x\right)\dd{s},
	\end{equation*}
	where $Q_n:= V_n (1-2q_n)$. Denote now $h_{n,\Psi} = n v^B_{n, \frac{\Psi}{n}}$ (which reflects the fact that the $n$-th approximation consists of the particles of size $1/n$ and the initial number of particles is $n$ times greater). It is easy to check that $Q_n = Q$ hence
% \begin{equation*}
% 	h_{n,\Psi}\left(x,r,t\right)=\int_{0}^{t}\mathcal{T}_{t-s}^Q \left[\Psi\left(\cdot,r+t-s\right)\left(1+\frac{1}{n} h_{n,\Psi}\left(\cdot,r+t-s,s\right)\right)+n V_n \rbr{\frac{1}{2}-\frac{Q}{4n}} \left(\frac{1}{n} h_{n,\Psi}\left(\cdot,r+t-s,s\right)\right)^2\right]\left(x\right)\dd{s}. 
% \end{equation*}
% This implies 
\begin{equation*}
	h_{n,\Psi}\left(x,r,t\right)=\int_{0}^{t}\mathcal{T}_{t-s}^Q \left[\Psi\left(\cdot,r+t-s\right)\left(1+\frac{1}{n} h_{n,\Psi}\left(\cdot,r+t-s,s\right)\right)+ \frac{V_n q_n}{n}  h_{n,\Psi}^2\left(\cdot,r+t-s,s\right)\right]\left(x\right)\dd{s}. 
\end{equation*}
We have $\frac{V_n q_n }{n} \rightarrow Vq $. By the convergence from Proposition \ref{prop:superprocess-convergence} (we use a different starting condition but it does not influence the convergence)
\begin{equation}
	h_{n,\Psi} \rightarrow v^S_{\Psi} \label{eq:v-n} 
\end{equation}
It is easy to check that $v^S_{\Psi}$ fulfils \eqref{eq:w-integral-super}.
\end{proof}
Equation \eqref{eq:w-integral-super} has a  useful series representation. Firstly, for $f,g \in \mathcal{B}(\mathbb{R}^{d+1})$ we define the convolution operator
\[
	(g*f)(x,t) = \intc{t} \T{t-s}^Q \sbr{g(\cdot,s) f(\cdot,s)}(x) \dd{s}.
\]
Let us fix $t_0>0$ and denote $v^S_\Psi(x,s) := v^S_\Psi(x,t_0-t,t)$. By \eqref{eq:w-integral-super} it fulfils 
\[	v^S_{\Psi}\left(x,t\right)=\int_{0}^{t}\mathcal{T}^Q_{t-s}\left[ \Psi(\cdot, t_0-s)  + Vq\left(v^S_{\Psi}\left(\cdot,s\right)\right)^2\right]\left(x\right)\dd{s}, \:\: t\in [0,t_0].
\]
Following the reasoning of \cite[Section 3]{Hong:2005aa} and assuming that $\norm{\Psi}{\infty}{} < Q_0^S$ it can be verified that
\begin{equation}
		v^S_{\Psi}(x,t) = (Vq)^{-1}\sum_{n=1}^{\infty } F^{*n}(x,t), \label{eq:rep}
\end{equation}
where $F^{*1}(x,t) := Vq \intc{t} \T{t-s}^Q \Psi(x,t_0-s) \dd{s}$ and
\[
		 F^{*n}(x,t) = \sum_{l=1}^{n-1}  (F^{*l}*F^{*(n-l)})(x,t), n\geq 2.
\]
The next lemma is an obvious modification of \cite[Lemma 3.1]{Hong:2005aa} 
\begin{lem}\label{lem:hong-infty}
	Assume that $\Psi \in \mathcal{B}(\Rd)_+$ then for $F^{*n}$ defined above we have
	\[
		\norm{F^{*n}}{\infty}{} \leq B_n Q^{1-2n} \norm{Vq \Psi}{\infty}{n},\:\: n\geq 1,
	\]
	\[
		\norm{F^{*n}}{\infty}{} \leq D_n Q^{1-2n} \norm{F^{*1}}{\infty}{n}, \:\: n \geq 1,
	\]
where $\cbr{B_n}_{n\geq 1}$ is the sequence defined by $B_1 = B_2 = 1$ and $B_n = \sum_{k=1}^{n-1} B_k B_{n-k}$ and the sequence  $\cbr{D_n}_{n\geq 1}$ is defined by $D_1 = Q, D_2 = Q^{-2}$ and $D_n = \sum_{k=1}^{n-1} D_k D_{n-k}$.
\end{lem}
It is easy to show \cite[proof of Lemma 3.2]{Hong:2005aa} that $ 5^{-n} B_n \rightarrow 0$ analogously there exists $A:=A(Q)>0$ such that $A^{-n} D_n \rightarrow 0$. Using the representation \eqref{eq:rep} we prove 
\begin{lem}\label{lem:helper}
	Assume that (A7) holds. For any $\varphi \in \SD_+$ there exist $C_1, C_2$ such that for any c\'al\'ag $\chi:\R \rightarrow \R_+$ such that $\norm{\chi}{\infty}{}<C_1$ and any $t_0>0$ we have
	\begin{equation*}
		\intc{t_0} \norm{v^S_{\Psi}(\cdot,t_0-s,s)}{1}{} \dd{s} \leq C_2 \norm{\varphi}{1}{} \intc{t_0} \chi(s) \dd{s},\quad \Psi(x,t):=\varphi(x)\chi(t). %\label{eq:helper}
	\end{equation*}
\end{lem}
\begin{proof}\label{pf:} Fix $a\in(0,1)$, by the representation \eqref{eq:rep} the proof will be concluded once we show that
	\begin{equation}
		\intc{t_0} \norm{F^{*n}(\cdot,s)}{1}{} \dd{s} \leq a^{-n} \norm{\varphi}{1}{} \intc{t_0} \chi(s) \dd{s}, \:\: n\geq 1. \label{eq:tmp16}
	\end{equation}
	Using the triangle and generalised Minkowski inequality we check that 
	\[
		\intc{t_0}\norm{F^{*n}(\cdot,s)}{1}{} \dd{s} \leq \sum_{k=1}^{n-1} 	\intc{t_0} \intc{t} \norm{\T{t-s}^Q \sbr{ F(\cdot,s)^{*k} F(\cdot,s)^{*(n-k)}}}{1}{}  \dd{s} \dd{t}.
	\]
	For $n\geq 3$ by (A7) and Lemma \ref{lem:hong-infty} we get
	\[
		\intc{t_0}\norm{F^{*n}(\cdot,s)}{1}{} \dd{s} \leq  2C \sum_{k=1}^{\lceil n/2 \rceil} B_{n-k} Q^{1-2(n-k)} \norm{Vq \Psi}{\infty}{n-k} \intc{t_0} \intc{t} e^{-Q'(t-s)} \norm{\sbr{ F(\cdot,s)^{*k}}}{1}{} \dd{s}  \dd{t}.
	\]
	Changing the order of integration we get
	\[
		\intc{t_0}\norm{F^{*n}(\cdot,s)}{1}{} \dd{s} \leq \frac{2C}{Q'} \sum_{k=1}^{\lceil n/2 \rceil} B_{n-k} Q^{1-2(n-k)} \norm{Vq \Psi}{\infty}{n-k} \intc{t_0}   \norm{\sbr{ F(\cdot,s)^{*k}}}{1}{} \dd{s}.
	\]
	Assume now that we already know that \eqref{eq:tmp16} is true for $k<n$ then
	\[
		\intc{t_0}\norm{F^{*n}(\cdot,s)}{1}{} \dd{s} \leq \underbrace{\frac{2C}{Q'} \sum_{k=1}^{\lceil n/2 \rceil} B_{n-k} Q^{1-2(n-k)} (C_1 Vq)^{n-k} \norm{\varphi}{\infty}{n-k} a^{-k}}_{w} \rbr{\norm{\varphi}{1}{} \intc{t_0} \chi(s) \dd{s}}.
	\]
	Using the fact that $B_n 5^{-n}\rightarrow 0$ we can find $C_1$ small enough to have $w<a^{-n}$. The cases $n=1,2$ can be checked directly, finally by appealing to the induction we finish the proof.
%	 
% -------------------	
% 	\[
% 		\intc{t_0}\norm{F^{*n}(\cdot,s)}{1}{} \dd{s} \leq \sum_{k=1}^{n-1} 	\intc{t_0} \intc{t} \norm{\T{t-s}^Q \sbr{ F(\cdot,s)^{*k} F(\cdot,s)^{*(n-k)}}}{1}{}  \dd{s} \dd{t} 
% 	\]
% 	
% 	
% 	Using assumption (A7) we check this for $n=1,2$. The rest will follow by induction
% 	\[
% 		\intc{t_0}\intr F^{*n}(x,s) \dd{x} \dd{s} = \sum_{k=1}^{n-1} 	\intc{t_0}\intr \intc{t} \T{t-s}^Q \sbr{ F(\cdot,s)^{*k} F(\cdot,s)^{*(n-k)}}(x)  \dd{s} \dd{x} \dd{t} 
% 	\]	
% 	Using Lemma \ref{lem:hong-infty} and assumption (A7) we get
% 	\[
% 		\intc{T}\intr F^{*n}(x,s) \dd{x} \dd{s} \leq 2 \sum_{i=1}^{\lceil n/2 \rceil} B_{n-k} Q^{1-2(n-k)} \norm{\Psi}{\infty}{n-k} c_2 \intc{T}\intr \intc{t} e^{-Q(t-s)} \sbr{ F(x,s)^{*k}} \dd{s} \dd{x} \dd{t}
% 	\]	
% 	
% 	
% 	----
% 	We choose $C_1$ small enough to  have $\norm{F(x,t)}{+\infty}{} \leq c$ ($c$ is to be chosen). %The proof of \eqref{eq:helper} will be concluded once we show that there exists $c_1>0$ and $q<1$ such that for any $n$ we have
% \[
% 	\intc{T}\intr F(x,s)^{*n} \dd{x} \dd{s} \leq c_1 q^{-n} \norm{\varphi}{1}{} \intc{T} \chi(s) \dd{s} 
% \]
% 
% 
% Integrating with respect to $t$ and using the induction assumption we get 
% \[
% 	2 \sum_{i=1}^{\lceil n/2 \rceil} B_{n-k} Q^{-2(n-k)} \norm{\Psi}{\infty}{n-k} c_2 c_1 q^{-k} \norm{\varphi}{1}{} \intc{T} \chi(s) \dd{s} 
% \]
% Choosing $C_1$ small enough and consequently $\norm{\Psi}{\infty}{}$ small enough finishes the proof.
\end{proof}
Let us also define 
\begin{equation}
	\tilde{v}_{\Psi}(x,r,t) = \intc{t} \T{t-s}^Q \Psi(\cdot,r+t-s) \dd{s}, \quad \Psi \in \mathcal{S}(\mathbb{R}^{d+1} ). \label{sol:tv} 
\end{equation}
%Intuitively $\tilde{v}_\Psi$ was obtained by dropping quadratic terms in (\ref{eq:v-integral}) which ``do not play role'' when $\Psi$ is small. The quality of the approximation is expressed in terms of function $u$ 
\begin{equation}
	u^B_{\Psi} := v^B_{\Psi} -\tilde{v}_{\Psi}, \quad u^S_{\Psi} := v^S_{\Psi} -\tilde{v}_{\Psi}. \label{eq:def-u} 
\end{equation}
We have 
\begin{lem}\label{lem:differentus} 
	$u^B_\Psi, u^S_\Psi$ satisfy the equations
	\begin{equation}
		u^B_{\Psi} (x,r,t) = \intc{t} \T{t-s}^Q \sbr{\Psi(\cdot, r+t-s)v_\Psi^B(\cdot, r+t-s,s) + Vq (v^B_{\Psi}(\cdot, r+t-s,s))^2}(x) \dd{s}. \label{eq:delta-v} 
	\end{equation}
	\begin{equation}
		u^S_{\Psi} (x,r,t) = \intc{t} \T{t-s}^Q \sbr{Vq  (v_{\Psi}^S(\cdot, r+t-s,s))^2}(x) \dd{s}. \label{eq:delta-v-branching} 
	\end{equation}
\end{lem}
For proof see \cite[Lemma 3.2]{Mios:2009oq}.
% \begin{proof}
% 	Subtracting equations \eqref{eq:v-integral} and \eqref{eq:tildev} we obtain 
% 	\begin{equation}
% 		u_{\Psi} (x,r,t) = \intc{t} \T{t-s} \sbr{-Q u_{\Psi}(\cdot, r+t-s,s) + V\ v_{\Psi}^2(\cdot, r+t-s,s)} \dd{s}. \label{eq:Delta-v-1} 
% 	\end{equation}
% 	Although we do not know solution of \eqref{eq:v-integral} we may treat $v_\Psi$ as a known function. It is easy to check that \eqref{eq:delta-v} solves \eqref{eq:Delta-v-1}. Standard application of the Banach contraction principle proves that that it is unique. 
% \end{proof}
We will now present two lemmas for estimation of $v^B_{\Psi}$. They will have a common proof. 
\begin{lem} \label{lem:norm-infty-v1} There exists $C_1,C_2>0$ such that for any $\Psi \in \mathcal{B}(\mathbb{R}^{d+1})_+$ such that $\norm{\Psi}{\infty}{}<C_1$ there is
	\[
		\norm{v^B_{\Psi}}{\infty}{} \leq C_2 \norm{\Psi}{\infty}{}.
	\]
\end{lem}
As a consequence we get that for fixed $\Psi\geq 0$ there are $C$ and $\theta_0$ such that 
\begin{equation}
	\norm{u^B_{\theta \Psi}}{\infty}{} \leq C \theta^2, \quad 0\leq \theta \leq \theta_0. \label{eq:u-estimate}
\end{equation}
Let now $\Psi$ be of the form $\Psi(x,t) =\varphi(t) 1_{[s,s+\delta]}(t)$ for some $s,\delta>0$. 
\begin{lem} \label{lem:norm-infty-v2} There exists $C_1,C_2>0$ such that for any $\varphi \in \mathcal{B}(\mathbb{R}^{d})_+$ such that $\norm{\varphi}{\infty}{}\delta<C_1$ there is
	\[
		\norm{v^B_{\Psi}}{\infty}{} \leq C_2 \norm{\varphi}{\infty}{}\delta.
	\]
\end{lem}
\begin{proof} We notice that it is enough to show the claim for $v(x,t) := v_\Psi^B(x,T-t,t)$, for any $T>0$ (with constants independent of $T$). Moreover it is upper-bounded by $L$, being the solution of
	\begin{equation*}
		L(x,t)=\int_{0}^{t}\mathcal{T}^Q_{t-s}\left[\norm{\Psi\left(\cdot,T-s\right)}{\infty}{} (1+L\left(\cdot,s\right)) + Vq  (L\left(\cdot,s\right))^2  \right]\left(x\right)\dd{s}.
	\end{equation*}
	Obviously the space parameter is now superfluous, skipping it we obtain yet simpler equations 
	\begin{equation*}
		L(t)=\int_{0}^{t}e^{-Q(t-s)} \left[\norm{\Psi\left(\cdot,T-s\right)}{\infty}{} (1+L\left(s\right)) + Vq  (L\left(s\right))^2  \right]\dd{s},
	\end{equation*}
	which in a differential form writes as 
	\[
		L'(t) = -Q L(t) + \norm{\Psi\left(\cdot,T-t\right)}{\infty}{} (1+L(t)) + Vq  (L(t))^2,\quad L(0)=0. 
	\]
	It is upper-bounded by the solution of 
	\[
		K'(t) = -(Q/2) K(t)+ \norm{\Psi\left(\cdot,T-t\right)}{\infty}{} (1+K(t)) ,\quad K(0)=0,
	\]
	as long as $\norm{K}{\infty}{} \leq \frac{Q}{2Vq}$. One checks that once we assume $\norm{\Psi}{\infty}{}\leq Q/4$ we have
	\[
		K(t)\leq\frac{\norm{\Psi}{\infty}{}}{\norm{\Psi}{\infty}{}-Q/2}\rbr{e^{(\norm{\Psi}{\infty}{}-Q/2)t}-1}\leq\frac{4\norm{\Psi}{\infty}{}}{Q}.
	\]
	Now one can easily choose $C_1$ such that Lemma \ref{lem:norm-infty-v1} holds. To see Lemma \ref{lem:norm-infty-v2} we notice that without loss of generality we may assume that 
	\[
		K'(t) = -(Q/2) K(t)+ \norm{\varphi}{\infty}{}1_{[0,\delta]}(t) (1+K(t)) ,\quad K(0)=0.
	\]
	The sup is attained for $K(\delta)$ hence we have
	\[
		K(t)\leq K(\delta) = \frac{\norm{\varphi}{\infty}{}}{\norm{\varphi}{\infty}{}-Q/2}\rbr{e^{(\norm{\varphi}{\infty}{}-Q/2)\delta}-1}\leq e^{\norm{\varphi}{\infty}{}\delta}-1.
	\]
	The final step is to choose $C_1$ small enough to have $e^{C_1}-1\leq \frac{Q}{2Vq}$.
\end{proof}

Due to the representation presented above $v^S_\Psi$ is easier to handle. It is useful to know that $v^B_\Psi$ is comparable with $v^S_\Psi$. This property allows to convert some proofs for the superprocess into proofs for the BPS semi-automatically. 
\begin{lem} \label{lem:comparison} There exists $C>0$ such that for any $\Psi \in \mathcal{B}(\mathbb{R}^{d+1})_+$ which fulfils the assumptions of Lemma \ref{lem:norm-infty-v1} or Lemma \ref{lem:norm-infty-v2} we have
	\begin{equation}
		v^B_{C\Psi}(x,r,t) \leq v^S_{\Psi}(x,r,t) \leq  v^B_{\Psi}(x,r,t).\label{eq:comparison}
	\end{equation}
\end{lem}
\begin{proof}
	The second inequality is easy and is left to the reader. We may choose $C>0$ such that $C(1+\norm{v^B_{C\Psi}}{}{})\leq 1$. Therefore
	\begin{multline*}
		v^B_{C\Psi}\left(x,r,t\right)=\int_{0}^{t}\mathcal{T}^Q_{t-s}\left[C\Psi\left(\cdot,r+t-s\right)(1+v^B_{C\Psi}(\cdot,r+t-s,s)) + Vq  (v^B_{C\Psi}\left(\cdot,r+t-s,s\right))^2  \right]\left(x\right)\dd{s} \\ \leq 
		\int_{0}^{t}\mathcal{T}^Q_{t-s}\left[\Psi\left(\cdot,r+t-s,s\right) + Vq  (v^B_{C\Psi}\left(\cdot,r+t-s,s\right))^2  \right]\left(x\right)\dd{s}.
	\end{multline*}
Easy application of the Banach contraction principle concludes the proof.
\end{proof} 
% 
% 
% \begin{proof}
% 	Without lost of generality we assume that $\norm{\Psi}{\infty}{}=1$. It is easy to check that $v^B_{\theta \Psi} \leq v^B_{\theta}$ (i.e. $\Psi=1$). We write $v_\theta(t)$ instead of $v_\theta^B(x,r,t)$ as $x,r$ do not play role. In this case \eqref{eq:v-integral} writes as 
% 	\[
% 		v_\theta'(t) = Vq v_\theta^2(t) - (Q-\theta) v_\theta(t) + \theta, \:\: v_\theta(0) = 0.
% 	\]
% 	This equation can be solved explicitly (see e.g. \cite[Lemma 3.3]{Hong:2005aa}). Assume that $\frac{\theta}{Vq}  < \frac{(Q-\theta)^2}{4 (Vq)^2}$ then
% 	\[
% 		v_\theta(t) = \frac{1}{(Vq)^2} \frac{2\theta (1-e^{-\gamma t})}{\rbr{\frac{Q-\theta}{Vq} +\gamma } - \rbr{\frac{Q-\theta}{Vq} -\gamma } e^{-\gamma t}} \leq \frac{1}{(Vq)^2} \frac{2\theta}{\rbr{\frac{Q-\theta}{Vq} +\gamma } },
% 	\]
% 	where $\gamma = \sqrt{{\frac{(Q-\theta)^2}{(Vq)^2}} -  \frac{\theta}{Vq} } $. The result follows by differentiation of the upper bound at $\theta =0$.
% \end{proof}

\begin{proof}[Proof of Lemma \ref{lem:convergence-v2}(sketch). ]\label{pf:s}
	Let us define operator 
	\[
		(F(v))(x,r,t) = \intc{t} \T{t-s}^Q \theta\varphi(x) + \intc{t} \T{t-s}^Q \rbr{\theta\varphi(\cdot,s) v(\cdot,s,\theta)+ Vq v^2(\cdot,s,\theta)}(x) \dd{s}.
	\]
	For $\theta>0$ the sequence $\rbr{0, F(0), F(F(0)), \ldots}$ is non-decreasing and by the proof of the previous lemma bounded. Therefore it converges to a solution of the equation. One can also check that for $t$ small enough the operator is contraction. The unicity can be proven easily by subtractions of two distinct solutions. 
\end{proof}

To keep the proofs comprehensive we utilise the following notation
\begin{equation}
	v_T^B(x,r,t) : =v^B_{\Psi_T}(x,r,t) \:\text{ and }\: v^B_T(x) := v^B_T(x,0,T), \label{eq:simpl1} 
\end{equation}
The same also applies to $u^B_{\Psi_T}, v^S_{\Psi_T}, u^S_{\Psi_T}$.
%It is easy to notice that it is the same equation as \eqref{eq:tildev} hence $\tilde{v} = v'(0)$ (note that the above calculation is not quite rigorous as one have to justify differentiation under integral in \eqref{eq:trata1}).

%, it should be treated heuristically. The strict proof could be done done using definition \eqref{def:v} by dirrerentating and then condtion on the first splitting time.
\subsection{Laplace transforms} \label{sec:laplace} This section we are going to compute the Laplace transforms of space time variables $\tilde{Y}_T$ and $\tilde{X}_T$ defined by \eqref{eq:space-time-method}  for \eqref{eq:rescaled-occupation} and \eqref{def:occupation-process}. Let us recall notation \eqref{def:simplification}. We start with the branching particle system $N^B$.% According to discussion in Introduction the initial distribution is set null. The immigration is the Poisson random field with intensity $H\rbr{\lambda \otimes \lambda},H>0$. We have 
\begin{prop} \label{prop:laplace-branching} 
	Let $\Phi \in \mathcal{B}(\mathbb{R}^{d+1})$ such that $\Psi\leq Q_0^B$ then for $\tilde{Y}^B_T$, $\tilde{X}^B_T$  defined for the BPS $N^B$ we have
	\[
		\lap{\ddp{\tilde{X}^B_T}{\Phi}} = \exp\rbr{H \intc{T}\intr u^B_T(x,T-s,s) \dd{s} },
	\]
	\[
		\lap{\ddp{\tilde{Y}^B_T}{\Phi}} = \exp\rbr{H \intc{T}\intr v^B_T(x,T-s,s) \dd{s} }.
	\]
\end{prop}
The proof is analogous to the proof in \cite[Section 3.3]{Mios:2009oq}. For the superprocess $N^S$ we have 
\begin{prop} \label{prop:laplace-super} 
	Let $\Phi\in \mathcal{B}(\mathbb{R}^{d+1})$ such that $\Psi\leq Q_0^S$ then for $\tilde{Y}^S_T$, $\tilde{X}^S_T$  defined for the superprocess $N^S$ we have
	\[
		\lap{\ddp{\tilde{X}^S_T}{\Phi}} = \exp\rbr{H \intc{T}\intr u^S_T(x,T-s,s) \dd{s} },
	\]
	\[
		\lap{\ddp{\tilde{Y}^S_T}{\Phi}} = \exp\rbr{H \intc{T}\intr v^S_T(x,T-s,s) \dd{s} }.
	\]
\end{prop}
\begin{proof} Let us fix $\Psi$ fulfilling the assumptions. In Section \ref{sec:approx} we defined the sequence $\cbr{N^n}$ approximating the superprocess $N^S$. We denote by $\tilde{Y}^n_T$ the space time variable for \eqref{eq:rescaled-occupation} defined for $N^n$. We have
	\begin{equation}
		\ddp{\tilde{Y}^n_T}{\Phi} = \frac{T}{F_T} \sbr{\intc{1}\ddp{N^n_{Ts}}{\Psi(\cdot,s)}\dd{s}} = {\intc{T}\ddp{N^n_{s}}{\Psi_T(\cdot,s)}\dd{s}}. \label{eq:laplace-derivation} 
	\end{equation}
 Let us recall notation \eqref{def:simpl} and denote 
	\begin{equation*}
		K^n_T(\Phi) := \ev{\exp\rbr{ \ddp{\tilde{Y}^n_T }{\Phi} } } = \ev{\exp\rbr{{\intc{T}\ddp{N^n_{s}}{\Psi_T(\cdot,s)}\dd{s}}} }. 
	\end{equation*}
 Conditioning with respect to $Imm^n$ (Poisson random field describing the immigration of the $n$-th approximation), using independence of evolution of particles ({branching Markov property}) and \eqref{def:v} for the BPS starting from one particle we obtain
	\begin{multline}
		{\ev{\rbr{\left.\exp{\rbr{\intc{T}\ddp{N^{n}_{s}}{\Psi_T(\cdot,s)}\dd{s}}}\right|Imm^n}}} =\\
		\prod_{(t,x) \in \widehat{Imm}^n} \ev{\exp{\rbr{\int_t^T\ddp{N^{x,t,n}_{s-t}}{\Psi_T(\cdot,s)}\dd{s}}}} = \prod_{(t,x) \in \widehat{Imm}^n}\rbr{v^B_{n, \frac{\Psi_T}{n}}(x,t,T-t)-1}, 
	\end{multline}
	where $\widehat{Imm}^n$ is a (random) set such that $\sum_{(t,x) \in\widehat{Imm}^n} \delta_{(t,x)} = Imm^n\: a.s.$ viz. $\delta_{(t,x)}$ corresponds to a particle which immigrate to the system at time $t$ to location $x$. By $N^{x,t,n}$ we denote the BPS starting from location $x$ at time $t$ adhering to the dynamics of the $n$-the approximation. Following the notation of the proof of Lemma \ref{prop:laplace-super} we have 
	\begin{equation*}
		\ev{\exp{\rbr{-\intc{T}\ddp{N^{n}_{s}}{\Psi_T(\cdot,s)}\dd{s}}}} = \ev{\exp\cbr{\ddp{Imm^n}{\log (v^B_{n,\frac{\Psi_T}{n}}(\cdot,\star,T-\star)-1}}}, 
	\end{equation*}
	where $\cdot$,$\star$ denote integration with respect to space and time, respectively. Taking into account distribution of $Imm^n$ we obtain 
	\begin{equation*}
			K_T^n(\Phi) = \ev{\exp{\rbr{\intc{T}\ddp{N^{n}_{s}}{\Psi_T(\cdot,s)}\dd{s}}}} = \exp\rbr{ H \intc{T} \intr n v^B_{n,\frac{\Psi_T}{n}}(x,T-t,t) \dd{x} \dd{t}}. 
	\end{equation*}
	By the convergence \eqref{eq:v-n} and Proposition \ref{prop:superprocess-convergence} we get 
	\begin{equation}
		K_T^n(\Phi) \rightarrow K_T(\Phi) = \exp\cbr{ H \intc{T} \intr v^S_{\Psi_T}(x,T-t,t) \dd{x} \dd{t}  }, \label{eq:laplace1} 
	\end{equation}
	where $K_T(\Phi) = \ev{\exp\rbr{ \ddp{\tilde{Y}^S_T }{\Phi} } }$. Simple calculations show that 
	% By the properties of the Laplace transform we have (recall also that $v_T'(0) =\tilde{v}_T$ - see \eqref{eq:vprim} and $v(0)=0$) 
	% \begin{equation*}
	% 	\ev{}\ddp{\tilde{Y}_T}{\Phi} = \left.\frac{d}{\dd{\theta}}\right|_{\theta=0} K_T(\theta \Phi) = { -H \intc{T} \intr \tilde{v}_{\Psi_T}(x,T-t,t) \dd{x} \dd{t} - L \intr \tilde{v}_{\Psi_T}(x,0,T) \dd{x} }. 
	% \end{equation*}
	% Now we can calculate the Laplace transform of $\tilde{X}_T$. Using definition of $u_T$ \eqref{eq:def-u}, simple fact that $\tilde{X}_T = \tilde{Y}_T - \ev{\tilde{Y}_T}$ we obtain 
	\begin{equation}
		L_T(\Phi):=\ev{\exp\cbr{\ddp{\tilde{X}^S_T}{\Phi}}} = \exp\cbr{ H \intc{T} \intr u^S_T(x,T-t,t) \dd{x} \dd{t}}. \label{eq:laplace-imm} 
	\end{equation}
\end{proof}

\subsection{Central limit theorem} \label{sec:scheme} In this section we present the proof of Theorem \ref{thm:clt1}. We follow closely the lines of the proof of \cite[Theorem 2.1]{Mios:2009oq}. To make it clear we present a general scheme first. Although the processes $X_T$ are signed-measure-valued it is convenient to regard them as $\SP$-valued. In this space one may employ a space-time method introduced in \cite{Bojdecki:1986aa} which together with Mitoma’s theorem constitute a powerful technique in proving weak functional convergence. 
\paragraph*{Convergence} From now on we will denote by $\tilde{X}^S_T$ a space-time variable (recall \eqref{eq:space-time-method} with $\tau=1$) defined for $X^S_T$ for the superprocess. To prove convergence of $\tilde{X}^S_T$ we will use the Laplace functional 
\begin{equation*}
	L_T(\Phi) = \ev{ \exp \rbr{-\ddp{\tilde{X}^S_T}{\Phi}}},\:\: \Phi\in \mathcal{S}(\mathbb{R}^{d+1})_+. 
\end{equation*}
For the limit process $X$ denote 
\begin{equation*}
	L(\Phi) = \ev{\exp \rbr{-\ddp{\tilde{X}^S}{\Phi}}},\:\: \Phi\in \mathcal{S}(\mathbb{R}^{d+1})_+. 
\end{equation*}
Once we have established convergence 
\begin{equation}
	L_T(\Phi) \rightarrow L(\Phi),\quad \text{ as } T\rightarrow +\infty, \:\:{\Phi\in \mathcal{S}(\mathbb{R}^{d+1})_+}. \label{lim:laplace} 
\end{equation}
we will obtain weak convergence $\tilde{X}_T\Rightarrow \tilde{X}$ and consequently $X_T\rightarrow_i X$. Two technical remarks should be made here. We consider only non-negative $\Phi$ of the first form described in Section \ref{sec:notation}. The procedure how to extend the convergence to any $\Phi$ is explained in \cite[Section 3.2]{Bojdecki:2006ab}. Another issue is the fact that $\ddp{\tilde{X}^S_T}{\Phi}$ is not non-negative. The usage of the Laplace transform in this paper is justified by the special (Gaussian) form of the limit. For more detailed explanation one can check also \cite[Section 3.2]{Bojdecki:2006ab}. As explained in \cite{Bojdecki:2007aa} due to the special form of the Laplace transform convergence \eqref{lim:laplace} implies also finite-dimensional convergence. We have
\[
	L_T(\Phi) = \exp\Bigg({-H \underbrace{\intc{T}\intr \bar{u}^S_T(x,T-s,s) \dd{s}}_{A(T)}  }\Bigg), \:\: \Phi \in \mathcal{S}(\mathbb{R}^{d+1})_+,
\]
where $\bar{u}^S_T:= \tilde{v}_T - \bar{v}^S_T$ is an analogue of \eqref{eq:delta-v-branching} defined by
\begin{equation}
	\bar{u}^S_{T} (x,r,t) = \intc{t} \T{t-s}^Q \sbr{V  (\bar{v}_T^S(\cdot, r+t-s,s))^2} \dd{s}. \label{eq:tmp14}
\end{equation}
and $\bar{v}^S_T$ is an analogue of \eqref{eq:w-integral-super} given by \eqref{eq:integral-super-simple}. The formula for $L(T)$ can be proved in the same as Proposition \ref{prop:laplace-super}. Note here that $\bar{v}^S_{T}$ and $\bar{u}^S_{T}$ are much alike $v_T$ and $u_T$ in \cite{Mios:2009oq}. It is worthwhile to mention that $\bar{v}^S_T$ is far easier to analyse than $v^S_T$ because we have obvious inequalities
 \begin{equation}
 	0\leq \bar{v}^S_T \leq \tv \leq \frac{C_{\Phi}}{F_T}. \label{ineq:tv>v} 
 \end{equation}
We can also use the following simple estimation 
\begin{equation}
	\bar{u}^S_T \leq \frac{C_\Phi}{F_T^2}. \label{eq:estimate2} 
\end{equation}
Our aim now is to calculate the limit of $A(T)$. To this end we replace $v^S_T$ with $\tilde{v}_T$ in \eqref{eq:tmp14} and calculate the limit for such changed expression.
\begin{equation}
	\tilde{A}(T) = V \intr\: \intc{T} \intc{t} \T{t-s}^Q\sbr{ \tilde{v}^2_T(\cdot, T-s,s)}(x) \dd{s} \dd{t} \dd{x}.\label{eq:tilde-a4} 
\end{equation}
$\tilde{A}(T)$ is the same as $\tilde{A}_4(T)$ in \cite[Section 3.3]{Mios:2009oq}. Therefore we have 
\begin{equation*}
	\lim_{T\rightarrow +\infty}\tilde{A}(T) = 2V \intc{1} \chi(1-v)^2 \dd{v} \int_{0}^{+\infty} \intr \mathcal{U}^Q \sbr{ \T{s}^Q\varphi(\cdot) \T{s}^Q \mathcal{U}^Q\varphi(\cdot) }(x) \dd{x} \dd{s} . 
\end{equation*}
Note that by assumptions (A5) the integral above is finite. We are left with estimation of $\tilde{A}(T) - A(T)$. By the definition of $\bar{u}_T^S$ and inequality \eqref{ineq:tv>v} we have 
\begin{equation*}
	|\tilde{A}(T) - A(T)| \leq 2V \intr\: \intc{T} \intc{t} \T{t-s}^Q\sbr{ \bar{u}^S_T\rbr{\cdot,T-s,s} \tilde{v}_T(\cdot, T-s,s)}(x) \dd{s} \dd{t} \dd{x}. 
\end{equation*}
Using \eqref{eq:estimate2} and \eqref{sol:tv} we write 
\begin{equation*}
	|\tilde{A}(T) - A(T)| \leq \frac{2V}{F_T^2} \intr\: \intc{T} \intc{t} \T{t-s}^Q\sbr{ \intc{s}\T{s-u}^Q \Psi_T(\cdot,T-s)}(x) \dd{u}\dd{s} \dd{t} \dd{x}. 
\end{equation*}
Using \eqref{def:simpl}, after simple calculations, we get 
\begin{equation*}
	|\tilde{A}(T) - A(T)| \leq \frac{2V}{F_T^3} \intr\: \intc{T} \intc{t} u \T{u}^Q\varphi(x) \dd{u} \dd{t} \dd{x}. 
\end{equation*}
Now, by using d'Hospital rule, it follows easily from assumption (A5) that 
\begin{equation*}
	|\tilde{A}(T) - A(T)| \rightarrow 0, \quad \text{ as }T\rightarrow +\infty.
\end{equation*}

\paragraph*{Tightness} Using additional assumptions (A8),(A9) the tightness can be proved utilising the Mitoma theorem \cite{Mitoma:1983aa}. It states that tightness of $\cbr{X_T}_T$ with trajectories in $C([0, 1], \SP)$ is equivalent to tightness of $\ddp{X_T}{\varphi}$, in $\mathcal{C}([0,\tau],\mathbb{R})$ for every $\varphi \in \mathcal{S}(\Rd)$. We adopt a technique introduced in \cite{Bojdecki:2006aa}. Recall a classical criterion \cite[Theorem 12.3]{Billingsley:1968aa}, i.e. a process $\ddp{X_T(t)}{\varphi}$ is tight if for any $t,s\geq0$ and constant $C>0$
\begin{equation}
 \ev{(\ddp{X_T(t)}{\varphi}- \ddp{X_T(s)}{\varphi})^4} \leq C (t-s)^2. \label{ineq:tightness}
\end{equation}
Following the scheme in \cite{Bojdecki:2006aa} we define a sequence $(\psi_n)_n$ in $\mathcal{S}(\mathbb{R})$, and $\chi_n(u) = \int_u^1 \psi_n(s) \dd{s}$ in a such way that 
\begin{equation*}
 \psi_n\rightarrow \delta_t - \delta_s,\quad 0\leq \chi_n \leq \mathbf{1}_{[s,t]}.\label{tmpp:chi-n}
\end{equation*}
Denote $\Phi_n=\varphi\otimes\psi_n$. We have
\[
 \lim_{n\rightarrow +\infty} \ddp{X_T}{\Phi_n} = \ddp{X_T(t)}{\varphi} - \ddp{X_T(s)}{\varphi}
\]
 thus by the Fatou lemma and the definition of $\psi_n$ we will obtain (\ref{ineq:tightness}) if we prove that
\[
 \ev{\ddp{\tilde{X}_T}{\Phi_n}^4} \leq C (t-s)^2,
\]
where $C$ is a constant independent of $n$ and $T$. From now on we fix $n$ and denote $\Phi:=\Phi_n$ and $\chi := \chi_n$. 
By properties of the Laplace transform we have
\[
 \ev{\ddp{\tilde{X}_T}{\Phi}^4} = \left.\frac{d^4}{d\theta^4}\right|_{\theta = 0}\!\! \ev{\eexp{-\theta \ddp{\tilde{X}_T}{\Phi}}}
\]
Hence the proof of tightness will be completed if we show 
\begin{equation*}
 \left. \frac{d^4}{d\theta^4}\right|_{\theta = 0}\!\! \ev{\eexp{-\theta \ddp{\tilde{X}_T}{\Phi}}} \leq C(t-s)^2. \label{eqss: laplace-tighness-end}
\end{equation*}
For the sake of brevity the detailed calculation are left for the reader.

\subsection{Large deviation principle} 
In this  section we present the proof of Theorem \ref{thm:ldp}. The proof of Theorem \ref{thm:ldp2} is very similar (some parts can be transformed directly and some with a help of Lemma \ref{lem:comparison}). Let us recall definition \eqref{eq:legendre-lambda} and denote
	\begin{equation*}
		\Lambda_\varphi (T,\nu) := T^{-1} \log \ev{} \exp\rbr{ T \intc{1} \ddp{Y^S_T(t) }{\varphi} \nu(dt) }.
	\end{equation*}
\begin{lem}\label{lem:convergence-in-ellis}
	 Let $\varphi \in \SD $ and $\nu \in BV(\R)$ such that $\norm{\varphi}{\infty}{}\norm{\chi_\nu}{\infty}{}  < Q_0^S $. Then we have
	\[
		\lim_{T\rightarrow +\infty} \Lambda_\varphi(T,\nu) = \Lambda_\varphi(\nu).
	\]
\end{lem}
\begin{proof} Let us recall notation \eqref{def:notation-T} and denote $\Psi_T(x,s) := \varphi(x) \chi_T(s)$. Let $N^x$ denote the superprocess starting from $N_0^x = \delta_x$ without immigration. Definition \eqref{def:v} yields
	\begin{equation*}
		v_{\Psi_T}^S (x, T(1-t), Tt) =  \ev{} \exp \rbr{ \intc{Tt} \ddp{N_s^x}{\varphi} \chi_\nu((1-t) + s/T) \dd{s}} -1.
	\end{equation*}
	By \cite[Lemma 4.1]{Hong:2005aa} it is finite. Without loss of generality we may assume that $\chi_\nu$ is a càdlàg.
	Let us denote now
	\begin{equation*}
		A(T) := \intc{Tt} \ddp{N_s^x}{\varphi} \chi((1-t) + s/T) \dd{s} \quad B(T) := \chi(1-t) \intc{Tt} \ddp{N_s^x}{\varphi} \dd{s}.
	\end{equation*}
	We have 
	\begin{equation*}
		|A(T) - B(T)| \leq \intc{Tt} \ddp{N_s^x}{\varphi} |\chi((1-t)+s/T) - \chi(1-t)| \dd{s}. 
	\end{equation*}
	By the dominated Lebesgue convergence theorem (by the sub-criticality of the branching law there is $\intc{\infty} \ddp{N_s^x}{\varphi} \dd{s}$ is finite a.s. ) we have 
	\begin{equation*}
		|A(T) - B(T)| \rightarrow 0, \:\: \text{a.s.} 
	\end{equation*}
	A next usage of the dominated convergence theorem yields
	\begin{equation}
		v_\varphi(x, \chi(1-t)) = \lim_{ T\rightarrow +\infty} v^S_{\Psi_T} (x, T(1-t), Tt), \label{eq:tmpll}
	\end{equation}
	where $v_\varphi$ is defined by \eqref{eq:vvv}. By Proposition \ref{prop:laplace-super} we get
		\[
			\Lambda_\varphi(T,\nu) = T^{-1} \rbr{H \intc{T} \intr v^S_{\Psi_T}(x,T-t,t) \dd{x} \dd{t} } =
			H \intc{1} \intr v_{\Psi_T}(x,T(1-t),Tt) \dd{x}.
		\]
		Appealing to \eqref{eq:tmpll} and the dominated Lebesgue theorem concludes. 
		% \begin{equation*}
		% 	\Lambda_\varphi(\nu) = \lim_{T\rightarrow +\infty} \Lambda(T,\nu) = H \intc{1} \intr v_\varphi(x,  \chi_\nu(t)) \dd{t} \dd{x} %\rbr{= H u_\varphi(\theta \nu)}
		% \end{equation*}
	%on  $(-\underline{\theta},\overline{\theta})$. Now appeal to \cite[Theorem 2.2.4]{Deuschel:1989oq} proves the the second convergence in Theorem \ref{thm:ldp}.
	%\TDD{The first convergence}
\end{proof}
%\begin{lem} \label{lem:lem2}
% For any $\theta < \theta_1 := Q^2/(4 \norm{\varphi}{}{} \sup \chi)$ there exist $C$ such that 
%\begin{equation}
% v_T(x, \theta) <C, \quad \forall_T
%\end{equation} 
%\end{lem}
\begin{proof} [Proof of Theorem \ref{thm:ldp}] \textbf{Upper bound} We follow a standard route of showing functional large deviation principle via studying multidimensional case. This is also emphasised by the use of the same notation as in the infinite dimensional case. Let us consider $\theta= (\theta_1, \theta_2,\cdots,\theta_n) \in \mathbb{R}^n$ and set 
	\begin{equation}
		\chi_\theta(s) := \sum_{i=1}^n  (\theta_i + \theta_{i+1} + \ldots +\theta_n) \mathbf{1}_{[\frac{i-1}{n}, \frac{i}{n})}(s). \label{eq:chi-theta} 
	\end{equation}
	It is straightforward to check that according to \eqref{def:measure-chi} $\chi_\theta = \chi_\nu$ for $\nu = \sum_{i=1}^{n} \theta_i \delta_{i/n}$. We also denote 
	\[
		\Lambda_\varphi(T,\theta) = T^{-1}\log \lap{T \sum_{i=1}^{n} \theta_i \ddp{Y_T(i/n)}{\varphi} }
	\]
	By Proposition \ref{prop:laplace-branching} and Lemma \ref{lem:convergence-in-ellis} for any $\theta \in \mathbb{R}^n$ such that $\sup_t \chi_\theta(t) \leq Q_0^S/ \norm{\varphi}{\infty}{}$ we have
	\[
		\Lambda_\varphi(\theta) := \liminf_{T\rightarrow +\infty} \Lambda(T,\theta) = H \intc{1} \intr v_\varphi(x, \chi_\theta(t)) \dd{x} \dd{t} = \frac{H}{n} \sum_{i=1}^n \intr v_\varphi(x, \theta_i + \theta_{i+1}+\cdots +\theta_n) \dd{x}.
	\]
	We check that $\Lambda$ is a convex function with respect to each $\theta_i$. Let us recall \eqref{eq:def-A2} and take $f\in \mathbb{R}^n$ $f=(f_1,f_2,\ldots,f_n)$ such that $0< f_i - (f_{i-1} + \cdots +f_1) < \frac{A}{nH}$. For such vector we can calculate the Legendre transform
	\[
		\Lambda_\varphi^*(f) = \sup_{\theta\in \R^n} (\ddp{f}{\theta} - \Lambda_\varphi(\theta)).
	\]
	Indeed, let $V_\varphi(\theta) := \intr v_\varphi(x,\theta) \dd{x}$. Using standard calculus we know that sup is attained at the solution of the following set of equations:
	\[
		0=\frac{\partial}{\partial \theta_j}\rbr{\ddp{f}{\theta} - \Lambda_\varphi(\theta)} = f_j - \frac{H}{n} \sum_{i=1}^j V'_\varphi(x, \theta_i +\cdots +\theta_n), \quad j\in \cbr{1,\cdots,n}.
	\]
	The vector $f$ was chosen in such a way that the solution of the above set of equations $\theta = (\theta_1, \theta_2, \ldots, \theta_n)$ is such that $\forall_{i} \theta_i+\ldots+\theta_n < Q_0^S/ \norm{\varphi}{\infty}{}$ (more details can be found in the proof of \cite[Theorem 4.1]{Hong:2005aa}). In other words, for this $\theta$ we have $\sup_t \chi_\theta(t) \leq Q_0^S/ \norm{\varphi}{\infty}{}$). This considerations together with \cite[Lemma 2.3.9]{Dembo:1998fu} entitle us to use the G\"artner-Ellis theorem (see e.g. \cite[Theorem 2.3.6]{Dembo:1998fu}) establishing the result in the finite-dimensional setting.\\
	% 
	% 
	% Consider now $j=1$, we know that $\theta_1 +\ldots +\theta_n \leq \delta$ hence we need $0\leq u_1 \leq \frac{a}{H}$. Similarly for $j=1$ we have 
	% \[
	% 	u_2 - U'_\varphi(x, \theta_1 +\cdots +\theta_n) - U'_\varphi(x, \theta_2 +\cdots +\theta_n) = u_2 - u_1 - U'_\varphi(x, \theta_2 +\cdots +\theta_n),
	% \]
	% Now we conclude that $0\leq u_2 - u_1 \leq  \frac{a}{H}$ and so on $0\leq u_i - (u_{i-1} + \cdots +u_1) \leq \frac{a}{H}$. 
	Now we are ready to prove the upper bound. It is suff`icient to show that for any $f \in \mathcal{C}_{0,A/H}$ and any $\delta>0$ 
	\[
		\liminf_T T^{-1} \pr{\ddp{Y_T}{\varphi}\in B(f,\delta)} \geq -  \Lambda_\varphi^*(f),
	\]
	where $B(f,\delta)$ denotes ball in $\mathcal{C}([0,1], \R)$. Let us denote % Denote $\underline{f}=\inf_{t\in[0,1]} f'(x)$ and $\overline{f}=\sup_{t\in [0,1]}$ then set $\epsilon_n = \frac{\delta}{4n}\min\rbr{\underline{f}, A/H-\overline{f}}$ and finally define set 
	\[
		O_{n,\epsilon} = \cbr{g\in \mathcal{C}([0,1], \R): g(i/n)\in (f(i/n)-\epsilon, f(i/n)+\epsilon), i\in \cbr{0,1,\ldots, n}}, \quad n\in \mathbb{N}, \epsilon>0
	\]
	One can check that there exists $n$ and $\epsilon$ such that $O_{n,\epsilon} \setminus B(f,\delta)$ contains only functions which are not increasing (i.e. for any such function $g$ we can find $s<t$ such that $h(s)>h(t)$). Obviously $Y_T$ is almost surely increasing hence
	\[
		\pr{\ddp{Y_T}{\varphi}\in B(f,\delta)} \geq  \pr{\ddp{Y_T}{\varphi}\in O_{n,\epsilon}} \geq (*)
	\]
	This reduced the problem to finite number of dimensions therefore
	\begin{multline*}
		(*) = \liminf_T T^{-1} \pr{\ddp{Y_T(i/n)}{\varphi}\in  (f(i/n)-\epsilon, f(i/n)+\epsilon), i\in\cbr{1,\ldots,n}} \\\geq 
		- \Lambda_\varphi^*((f(0), f(1/n), \ldots, f(1 - 1/n))).
	\end{multline*}
	% \begin{multline*}
	% 		%\liminf_T T^{-1} \pr{\ddp{Y_T}{\varphi}\in B(x,\delta)} \geq \liminf_T T^{-1} \pr{\ddp{Y_T}{\varphi}\in O_{n,\epsilon}} = \\
	% 		 (*) = \liminf_T T^{-1} \pr{\ddp{Y_T(i/n)}{\varphi}\in  (f(i/n)-\epsilon, f(i/n)+\epsilon), i\in\cbr{1,\ldots,n}} \geq 
	% 		- \Lambda_\varphi^*((f(0), f(1/n), \ldots, f(1 - 1/n))). 
	% 	\end{multline*}
	Notice that the last quantity is the same as \eqref{eq:crippled-legendre} if we restrict $B$ in its definition to the set of point measures with the support in the set $\cbr{1/n,2/n,\ldots,1}$.
	\paragraph{Lower bound} % (fold)
	\label{par:lower_bound} Let us take function in $f \in \mathcal{C}_{0,A/H}$, for any $\delta>0$ we can find $\nu_0$ such that $\ddp{f}{\nu_0} - \Lambda_\varphi(\nu_0 ) \geq \Lambda_\varphi^*(f) - \frac{\delta }{4},$ and $\sup_{t} \chi_{\nu_0}(t) < Q_0^S/ \norm{\varphi}{\infty}{}$. $\nu_0$ can be approximated by a point measure as in the previous section such that  
	\[
		\ddp{f}{\nu} - \Lambda_\varphi(\nu ) \geq \Lambda_\varphi^*(f) - \frac{\delta }{2}.
	\]
	Following the notation of the previous section we write its total variation $|\nu|=\sum_{i=1}^n |\theta_i| <+\infty$ (as each $|\theta_i|$ is bounded; for details see the proof of \cite[Theorem 4.1]{Hong:2005aa}). We define $r := \delta/(2|\nu|)$ and consider a ball $B(f,r)$. For any $g\in B(f,r)$ we have $\ddp{\nu}{f-g} \leq \delta/2$. Using the Chebyshev inequality we obtain
	\begin{multline*}
		\frac{1}{T} \log \pr{\ddp{Y_T}{\varphi} \in \overline{B(f,r)}} \\\leq \frac{1}{T} \log \pr{  \intc{1} \rbr{\ddp{Y_T(s)}{\varphi} -f(s)} \nu(\dd{s}) \geq -\delta/2} = \frac{1}{T} \log \pr{ \intc{1} \ddp{Y_T(s)}{\varphi}\nu(\dd{s}) \geq \ddp{\nu}{f}-\delta/2} \\\leq 
		\frac{1}{T} \log \pr{ \exp\cbr{T \intc{1} \ddp{Y_T(s)}{\varphi} \nu(\dd{s})} \geq \exp\cbr{\ddp{T \nu}{f}-T \delta/2}} \leq \Lambda_\varphi(T,\nu) - \ddp{\nu}{f} + \delta/2
	\end{multline*}
	Now we easily conclude that
	\[
		\limsup_{T\rightarrow +\infty}\frac{1}{T} \log \pr{Y_T \in \overline{B(f,r)}} \leq - \Lambda^*(f) + \delta. 
	\]
	%The proof can be easily concluded by noticing that any compact set $K \subset \mathcal{C}_{a,b}^r$ can be covered by a finite number of balls as above.
\end{proof}

\subsection{Functional moderate deviation principle} \label{sec:mdp}  In this section we prove Theorem \ref{thm:ldp22} (we skip the proof of Theorem \ref{thm:ldp23} which is simpler). Throughout the whole proof $F_T = T^{(1+\alpha)/2},\: 0<\alpha<1$ and $\varphi\in \SD_+$ is fixed. We will skip the superscript $^B$. 
\paragraph*{Lower bound}
Firstly we will prove the lower estimate for compact set. Let $\nu \in BV(\R)$  and define 
\begin{equation*}
	\Lambda_\varphi(T,\nu) := T^{-\alpha}\log \rbr{\ev{}\exp\cbr{\theta T^{\alpha} \intc{1} \ddp{X_T(t)}{\varphi}  \nu(\dd{t})} }. 
\end{equation*}
We introduce an additional parameter $ \theta $; in this part of proof it is always $ \theta = 1 $. By Proposition \ref{prop:laplace-branching} we have
\begin{equation}
	\Lambda_\varphi(T,\nu) = T^{-\alpha} { H\intc{T} \intr u_{\Psi(T,\theta)}(x,T-t,t) \dd{x} \dd{t} }. \label{eq:Lambda}
\end{equation}
where $\Psi(T,\theta) = \theta T^{(\alpha-1)/2} \varphi \otimes \chi_T$ (see notation in Section \ref{sec:notation}). Since $\Psi(T,\theta)\rightarrow_T 0$ we know that $\Lambda_\varphi(T,\nu)$ is well-defined for $T$'s large enough.
\[ \Lambda_\varphi(T,\nu) = H T^{-\alpha} \intc{T} \intr u_{\Psi(T,\theta)}(x,T-t,t) \dd{x} \dd{t}. \]
Using equation \eqref{eq:delta-v} we obtain
\[ \Lambda_\varphi(T,\nu) = \Lambda_1(T,\nu) + \Lambda_2(T,\nu), \]
where 
\begin{equation*}
	\Lambda_1(T,\nu) := H T^{-\alpha} \intc{T} \intr \intc{t} \T{t-s}^Q \sbr{{\Psi(T,\theta)}(\cdot, T-s) v_{\Psi(T,\theta)}(\cdot, T-s,s) }(x) \dd{s} \dd{x} \dd{t} , \label{eq:Lambda1}
\end{equation*}
\begin{equation}
	\Lambda_2(T,\nu) := H Vq T^{-\alpha} \intc{T} \intr \intc{t} \T{t-s}^Q \sbr{ v_{\Psi(T,\theta)}^2(\cdot, T-s,s) }(x) \dd{s} \dd{x} \dd{t}.\label{eq:Lambda2}
\end{equation}
Using the Fubini theorem we get 
\begin{equation*}
	\Lambda_1(T,\nu) = H T^{-\alpha} \intc{T} \intr \intc{T-s} \T{u}^Q \sbr{{\Psi(T,\theta)}(\cdot, T-s) v_{\Psi(T,\theta)}(\cdot, T-s,s) }(x) \dd{u}  \dd{x}\dd{s},   
\end{equation*}
\begin{equation*}
	\Lambda_2(T,\nu) = H Vq T^{-\alpha} \intc{T} \intr \intc{T-s} \T{u}^Q  \sbr{{ v_{\Psi(T,\theta)}^2(\cdot, T-s,s) }}(x) \dd{u} \dd{x} \dd{s}.   
\end{equation*}
This is approximated by (we recall \eqref{eq:potentialOperator})
\begin{equation*}
	{\Lambda}_{1a}(T,\nu) := H T^{-\alpha} \intc{T} \intr \mathcal{U}^Q\sbr{{\Psi(T,\theta)}(\cdot, T-s) v_{\Psi(T,\theta)}(\cdot, T-s,s) }(x) \dd{x}\dd{s}, \label{eq:Lambda1a}
\end{equation*}
\begin{equation}
	{\Lambda}_{2a}(T,\nu) := H Vq T^{-\alpha} \intc{T} \intr \mathcal{U}^Q \sbr{ v_{\Psi(T,\theta)}^2(\cdot, T-s,s) }(x) \dd{x} \dd{s} . \label{eq:Lambda2a}
\end{equation}
In the next step we approximate $ v_{\Psi(T,\theta)}(x,T-s,s) $ with $ \intc{s} \T{s-u}^Q {\Psi(T,\theta)} \dd{u} = \theta T^{(\alpha-1)/2} \intc{s} \T{s-u}^Q \varphi(x)\chi_T(T-u) \dd{u}$, namely
\begin{equation*}
	\Lambda_{1b}(T,\nu) := H \theta^2 T^{-1} \intc{T} \intr \mathcal{U}^Q \sbr{ \varphi(\cdot) \chi_T(T-s) \intc{s} \T{s-u}^Q \varphi(\cdot)\chi_T(T-u) \dd{u}}(x) \dd{x} \dd{s},  \label{eq:Lambda1b}
\end{equation*}
\begin{equation}
	\Lambda_{2b}(T,\nu) := H Vq \theta^2 T^{-1} \intc{T} \intr  \mathcal{U}^Q \sbr{ \rbr{\intc{s} \T{s-u}^Q \varphi(\cdot)\chi_T(T-u) \dd{u}}^2}(x) \dd{x} \dd{s}. \label{eq:Lambda2b}
\end{equation}
Now we substitute $ s\rightarrow Ts $, use the Fubini theorem and \eqref{def:notation-T}
\begin{equation*}
	\Lambda_{1b}(T,\nu) = \theta^2 H \intc{1} \intc{Ts} \intr \mathcal{U}^Q \sbr{ \varphi(\cdot) \T{Ts-u}^Q \varphi(\cdot)\chi(1-u/T) \chi(1-s)}(x) \dd{x}\dd{u} \dd{s}, 
\end{equation*}
\begin{equation*}
	\Lambda_{2b}(T,\nu) = \theta^2 H Vq \intc{1} \intr \mathcal{U}^Q \sbr{ \rbr{\intc{Ts} \T{Ts-u}^Q \varphi(\cdot)\chi(1-u/T) \dd{u}}^2 }(x) \dd{x}\dd{s}. 
\end{equation*}
Substituting $ u\rightarrow Ts-u $ we get 
\begin{equation*}
	\Lambda_{1b}(T,\nu) = \theta^2 H \intc{1} \intc{Ts} \intr \mathcal{U}^Q \sbr{ \varphi(\cdot) \T{u}^Q \varphi(\cdot)}(x) \chi(1-s +u/T)\chi(1-s)  \dd{x} \dd{u} \dd{s},  
\end{equation*}
\begin{equation*}
	\Lambda_{2b}(T,\nu) = \theta^2 H Vq \intc{1} \intr \mathcal{U}^Q \sbr{ \rbr{\intc{Ts} \T{u}^Q \varphi(\cdot)\chi(1-s + u/T) \dd{u}}^2}(x)  \dd{x} \dd{s}. 
\end{equation*}
We define also 
\begin{equation*}
	\Lambda_{1c}(T,\nu) := \theta^2 H \intc{1} \intc{Ts} \intr \mathcal{U}^Q \sbr{\varphi(\cdot) \T{u}^Q \varphi(\cdot)}(x) \chi^2(1-s) \dd{x} \dd{u} \dd{s},  \label{eq:Lambda1c}
\end{equation*}
\begin{equation}
	\Lambda_{2c}(T,\nu) := \theta^2 HVq \intc{1} \intr \mathcal{U}^Q \sbr{\rbr{\intc{Ts} \T{u}^Q \varphi(\cdot)\chi(1-s) \dd{u}}^2}(x) \dd{x} \dd{s}.  \label{eq:Lambda2c}
\end{equation}
Finally we notice that both function are increasing in $T$ and recall assumption (A4) to get (we denote $T_3(\chi)= \intc{1} \chi(1-s)^2\dd{s}$)
\begin{equation*}
	\lim_{T\rightarrow +\infty }\Lambda_{1c}(T,\nu) = (\theta^2 H) T_1(\varphi) T_3(\chi), \quad \lim_{T\rightarrow +\infty }\Lambda_{2c}(T,\nu) = (\theta^2 HVq) T_2(\varphi) T_3(\chi). 
\end{equation*}
We will show that in the limits of $\Lambda_{1c}, \Lambda_{2c} $ are the same as the ones of $\Lambda_{1}, \Lambda_{2} $. Firstly, using (A7) we easily get 
\begin{multline*}
	| \Lambda_{2c}(T,\nu) - \Lambda_{2b}(T,\nu)| \leq \theta^2 C \intc{1} \intr \rbr{\intc{Ts} \T{u}^Q \varphi(x)|\chi(1-s) + \chi(1-s+u/T)| \dd{u}}\\ \rbr{\intc{Ts} \T{u}^Q \varphi(x)|\chi(1-s) - \chi(1-s+u/T)| \dd{u}} \dd{x} \dd{s}. 
\end{multline*}
Using (A7) the first integral is finite, hence 
\begin{equation*}
	| \Lambda_{2c}(T,\nu) - \Lambda_{2b}(T,\nu)| \leq  \theta^2 C_1 \intc{1} \intr \rbr{\intc{Ts} \T{u}^Q \varphi(x)|\chi(1-s) - \chi(1-s+u/T)| \dd{u}} \dd{x} \dd{s} \rightarrow 0.
\end{equation*}
since we can  observe that $ \norm{\chi}{\infty}{}<+\infty $ and one can use the Lebesgue dominated convergence theorem. Analogously
\begin{equation*}
	| \Lambda_{1c}(T,\nu) - \Lambda_{1b}(T,\nu)| \rightarrow 0. 
\end{equation*}
Using assumption (A7) again we have
\begin{equation*}
	|{\Lambda}_{1a}(T,\nu) - {\Lambda}_{1b}(T,\nu) | \leq C_1 T^{-\alpha} \intc{T} \intr {{|\Psi|(T,\theta)}(x, T-s) u_{|\Psi|(T,\theta)}(x,T-s,s)  }\dd{x} \dd{s}. 
\end{equation*}
By \eqref{eq:u-estimate} we know that for $T$'s large enough $ u_{|\Psi|(T,\theta)}(x,T-s,s) \leq C  T^{-1+\alpha}  $ hence 
\begin{equation*}
	|{\Lambda}_{1a}(T,\nu) - {\Lambda}_{1b}(T,\nu) | \leq C_2  T^{-1} \intc{T} \intr {{|\Psi|(T,\theta)}(x, T-s)  }\dd{s} \dd{x} \leq C_2 \theta^3 T^{(\alpha-1)/2} \rightarrow 0 .
\end{equation*}
Similarly using assumption (A7) and Lemma \ref{lem:norm-infty-v1} we get
\begin{equation*}
	|{\Lambda}_{2a}(T,\nu) - {\Lambda}_{2b}(T,\nu) | \leq C_1 T^{-\alpha} \intc{T} \intr u_{|\Psi|(T,\theta)}(x,T-s,s) v_{|\Psi|(T,\theta)}(x,T-s,s)  \dd{x} \dd{s} , 
\end{equation*}
\begin{equation*}
	|{\Lambda}_{2a}(T,\nu) - {\Lambda}_{2b}(T,\nu) | \leq C_2 T^{-1} \intc{T} \intr  v_{|\Psi|(T,\theta)}(x,T-s,s)  \dd{x} \dd{s}   \leq C_2 \theta^3 T^{(\alpha-1)/2} \rightarrow 0.
\end{equation*}
Once again we utilise (A7) to obtain
\[
	|\Lambda_{1a}(T,\nu) - \Lambda_{1}(T,\nu)| \leq C_1 T^{-\alpha} \intc{T} \intr {e^{-(T-s)Q'}}{{|\Psi|(T,\theta)}(x, T-s) v_{|\Psi|(T,\theta)}(x, T-s,s) } \dd{s} \dd{x}.  
\]
By Lemma \ref{lem:norm-infty-v1} we get 
\[
	|\Lambda_{1a}(T,\nu) - \Lambda_{1}(T,\nu)| \leq C_2 T^{-1} \intc{T} \intr {e^{-(T-s)Q'}}{{|\Psi|(T,\theta)}(x, T-s) } \dd{s}  \dd{x}  \leq C_3 T^{-1}\rightarrow 0.
\]
Analogously using Lemma \ref{lem:norm-infty-v1} once more  and then Lemma \ref{lem:helper} we have 
%\[
%	|\Lambda_{2a}(T,\nu) - \Lambda_{2}(T,\nu)| \leq C_1 T^{-\alpha} \intc{T} \intr {e^{-(T-s)Q}} { v_{|\Psi|(T,\theta)}^2(x, T-s,s) } \dd{s}  \dd{x} 
%\]
\[
	|\Lambda_{2a}(T,\nu) - \Lambda_{2}(T,\nu)| \leq 	H T^{-1} \intc{T} \intr {e^{-(T-s)Q'}} { v_{|\Psi|(T,\theta)}(x, T-s,s) } \dd{s}  \dd{x}\leq 	C T^{(\alpha-1)/2}  \rightarrow 0.
\]
Finally, we put all the calculation together
\begin{equation}
	\Lambda_\varphi(\nu):=\lim_{T\rightarrow +\infty} \Lambda_\varphi(T,\nu) = \theta^2 H \rbr{\intc{1} \chi(1-s)^2\dd{s}} \rbr{T_1(\varphi)+ Vq T_2(\varphi)}. \label{eq:tmp17} 
\end{equation}
Once we prove that the exponential tightness holds - Section \ref{sec:exp-tighntess} - by \cite[Theorem 2.2.4]{Deuschel:1989oq} and \cite[Lemma 1.3.8]{Deuschel:1989oq} lower bound in Theorem \ref{thm:ldp22} will be established.

\paragraph*{Upper bound} 
We start with the multidimensional case. We utilise the notation introduced in the proof of the large deviation principle. Namely, consider a vector $\theta= (\theta_1, \theta_2,\cdots,\theta_n)$ and recall \eqref{eq:chi-theta}. Further we denote
\[
	\Lambda_\varphi(T,\theta) := T^{-\alpha}\log \lap{T^{\alpha} \sum_{i=1}^{n} \theta_i \ddp{X_T(i/n)}{\varphi} }.
\]
This is in fact a special case of \eqref{eq:tmp17} hence we already know that
\[
	\Lambda_\varphi(\theta) = \lim_{T\rightarrow +\infty}\Lambda_\varphi(T,\theta) = H \rbr{\intc{1} \chi_\theta(1-s)^2\dd{s}}  \rbr{T_1(\varphi)+ Vq T_2(\varphi)}.
\]
Denote its Legendre transform by $\Lambda^*_\varphi$. \cite[Lemma 2.3.9]{Dembo:1998fu} entitle us to use the G\"artner-Ellis theorem (see e.g. \cite[Theorem 2.3.6]{Dembo:1998fu}) hence for any open set $O\subset \mathbb{R}^n$ we get the following deviation principle 
\[
	\liminf_{T\rightarrow +\infty} T^{-\alpha} \log \pr{(\ddp{X_T(1/n)}{\varphi}, \ldots,\ddp{X_T(1)}{\varphi}) \in O } \geq -\inf_{x \in O} \Lambda_\varphi^*(x),
\]
In order to prove \eqref{eq:mdp-upper} it suffices to show that for any $f \in H^1$  and $\epsilon>0$ 
\begin{equation}
\liminf_{T\rightarrow+\infty} T^{-\alpha} \log \pr{\ddp{X_T}{\varphi} \in B(f,\epsilon)} \geq - \Lambda_\varphi^*(f).	 \label{eq:get-it-rigtht}
\end{equation}
where $B(f,\epsilon)$ denotes a ball in $\mathcal{C}([0,1],\R)$. Consider now
\[
	O_n := \cbr{g\in \mathcal{C}([0,1],\R): g(i/n)\in (f(i/n)-\epsilon/2, f(i/n)+\epsilon/2), i\in\cbr{1,\ldots n}},  .
\]
\[
	\tilde{O}_n := \Pi_{i=1}^{n}  (f(i/n)-\epsilon/2, f(i/n)+\epsilon/2) \subset \R^n
\]
It is easy to check that for any $f\in O_n$ one have $\Lambda_\varphi^*(f) \geq \Lambda_\varphi^*(f(1/n), f(2/n), \ldots, f(1))$ hence
\[
	\liminf_{T\rightarrow +\infty} T^{-\alpha}\log \pr{\ddp{X_T}{\varphi} \in O_n } \geq -\inf_{x \in \tilde{O}_n} \Lambda_\varphi^*(x) \geq -\inf_{g\in O_n} \Lambda_\varphi^*(g)\geq - \Lambda_\varphi^*(f),
\]
To finish the proof we show that for any $C>0$ there exists $n$ such that 
\[
	\limsup_{T\rightarrow +\infty} T^{-\alpha} \log \pr{\ddp{X_T}{\varphi} \in O_n \backslash B(f,\epsilon)} \leq  -C. 
\]
Using the upper bound estimate we have only to prove that $\inf_{ f\in  cl\rbr{O_n  \backslash B(f,\epsilon)}}\Lambda_\varphi^*(f)>C$. To this end we choose $n$ such that $w(f,1/n)<\epsilon/10$ ($w$ being defined by \eqref{eq:modulus}). If  $f\in  cl\rbr{O_n  \backslash B(f,\epsilon)}$ then there exist $s,t$ such that $|s-t|\leq1/n$ and $|f(s)-f(t)|>\epsilon/5$. Using the Jensen inequality it is easy to show that $\int_s^t f'(u)^2 \dd{u} \geq \frac{\epsilon^2 n}{10}$.

\subsubsection{Exponential tightness} \label{sec:exp-tighntess}
Let us fix $\varphi \in \SD_+$ and recall that $F_T = T^{(1+\alpha)/2}$.  In this section we will prove exponential tightness of $\cbr{\ddp{X_T}{\varphi}}_T$. To keep notation short we write
\[
	x_T(t) := \ddp{X_T(t)}{\varphi},\: t\in[0,1].
\]
\begin{rem} \label{rem:smaller} 
	If the exponential tightness holds for $x$ defined with some $\varphi$ it is also true for $x_T$ defined with $C\varphi$ for any $C>0$. Therefore we are entitled to decrease $\varphi$ (finitely many times) if necessary.
\end{rem}
By $ST(x_T)$ we denote stopping times relative to the natural filtration of $x_T$. In our context  
\begin{lem}\label{thm:tight}
	Assume that for all $\lambda>0$ we have \\
	\begin{equation}
		\lim_{\eta \rightarrow +\infty} \limsup_{T\rightarrow +\infty} T^{-\alpha} \log \pr{\sup_{t\in[0,1]} |x_T(t)| \geq \eta } = -\infty, \label{eq:exp-tight1} 
	\end{equation}
	\begin{equation}
		\lim_{\delta \rightarrow 0}  \limsup_{T\rightarrow +\infty} \sup_{\tau \in ST\rbr{x_T}} T^{-\alpha} \log \pr{\sup_{t\leq \delta} |x_T((\tau+t)\wedge 1) - x_T(\tau)|\geq \lambda} = -\infty. \label{eq:exp-tight2} 
	\end{equation}
	then sequence $\cbr{x_T}_T$ is exponentially tight.
\end{lem}
This follows easily from \cite[Theorem 3.1]{Liptser:2005vf}. Let us denote now total fluctuation of occupation time by $x$ (i.e. we can take $x(t) := x_{t}(1)$ with $F_T=1$ as a definition). We will need the following estimate
\begin{lem} \label{lem:estimate-xyz}
	Let $\epsilon\in (0,1/2)$ then there exist $c>0$ and $T_0>0$ such that for $T>T_0$ and $0 \leq \theta' \leq T^{-\epsilon}$
	\begin{equation}
		\lap{\theta' \sbr{x(t) - x(s)} } \leq \exp \rbr{ c (t-s) (\theta')^2  }, \:\: 0 \leq s<t \leq T \label{eq:laplace-xyz} 
	\end{equation}
\end{lem}
\begin{proof} The proof is based upon the proof of the lower bound from Section \ref{sec:mdp}. Given $\epsilon>0$ we put ${\alpha} := {1-2\epsilon}$ and appropriate $F_T = T^{1-\epsilon}$. Further we denote $\Theta := T^\epsilon \theta'$. By assumptions we have $\Theta \in [0,1]$. Let us recall \eqref{eq:Lambda}. We write 
	\[
		\lap{\theta \sbr{x(a+b) - x(a)} } = \lap{  T^{1-2\epsilon} \Theta \sbr{x_T((a+b)/T)) - x_T(b/T)}} =\exp\rbr{ T^{\alpha}\Lambda_\varphi(T,\nu) },
	\]
	% \begin{multline*}
	% 	\lap{\theta \sbr{x(t) - x(s)} } = \lap{  T^{1-2\epsilon} \Theta \sbr{x_T(t/T)) - x_T(t/T)}} =\exp\rbr{ T^{\alpha}\Lambda_\varphi(T,\nu) },
	% %	\lap{T^{\alpha} \Theta\sbr{x_T((t+s)/T)) - x_T(t/T)}} = 
	% \end{multline*}
	where $\nu = \delta_{(a+b)/T} - \delta_{a/T}$. Let us denote 
	\[
		H(T,\Theta, b) := \Theta^2 H \frac{b}{T} \rbr{T_1(\varphi) + Vq T_2(\varphi)} = \underbrace{H T^{-\alpha} \theta^2 b T_1(\varphi)}_{H_1(T,\Theta, b)} + \underbrace{HVq T^{-\alpha} \theta^2 b T_2(\varphi)}_{H_2(T,\Theta, b)}.
	\]
 We claim that $\Lambda_{\varphi}(T,\nu) \approx c H(T,\theta, b)$. To be more precise the lemma will be shown once we prove
	\[
		\limsup_{T\rightarrow +\infty}  \sup_{\Theta \in (0,1)} \sup_{0<a<a+b<T} \frac{\Lambda_\varphi(T,\nu) - H(T,\Theta, b) }{H(T,\Theta, b)} = C<+\infty.
	\]
We start with considering $\frac{\Lambda_2(T,\nu) - H_2(T,\Theta, b) }{H(T,\Theta, b)}$. In this direction we going to use the chain of approximations of \eqref{eq:Lambda} from Section \ref{sec:mdp}. That is $\Lambda_{2},\Lambda_{2a},\Lambda_{2b},\Lambda_{2c}$ given by \eqref{eq:Lambda2}-\eqref{eq:Lambda2c}. Obviously, by the fact that $\varphi\geq 0$ and $\chi = \chi_\nu \geq 0$ (recall \eqref{def:measure-chi}), we have $\Lambda_{2}(T)\leq \Lambda_{2a}(T)$, so we have only check the rest of the terms. Let us recall that $\Psi(T,\Theta) = \Theta T^{(\alpha-1)/2} \varphi \chi_T$. Using definitions \eqref{sol:tv}, \eqref{eq:def-u}, assumption (A7), inequalities \eqref{eq:u-estimate}, $\tilde{v}_{\Psi(T,\theta)} \leq v_{\Psi(T,\theta)}$, \eqref{eq:u-estimate} and finally Lemma \ref{lem:helper} we prove
\begin{multline*}
	\frac{\Lambda_{2a}(T,\nu) - \Lambda_{2b}(T,\nu) }{ H(T,\Theta,b)  } \leq 
	\frac{   C_1 T^{-\alpha} \intc{T} \intr \mathcal{U}^Q\sbr{u_{\Psi(T,\theta)}(x,T-s,s) v_{\Psi(T,\theta)}(x,T-s,s)}  \dd{s}  \dd{x}  }{C \Theta^2  T^{-1} b  } \leq\\
	\frac{   C_2 \Theta^2  \intc{T} \intr v_{\Psi(T,\theta)}(x,T-s,s)  \dd{s}  \dd{x}  }{\Theta^2 b  } \leq C_3 T^{(\alpha-1)/2} b^{-1} \intc{T}\chi_T(T-s) \dd{s} = C_3 T^{(\alpha-1)/2}  \rightarrow 0.
\end{multline*}
Further using assumption (A7) again we derive
	\[
		\frac{|\Lambda_{2b}(T,\nu) - \Lambda_{2c}(T,\nu) |}{ H(T,\Theta,t) } \leq \frac{C_1 \Theta^2 \intc{1} \intr \left| {\rbr{\intc{Ts} \T{u}^Q \varphi(x)\chi(1-s) \dd{u}}^2 - \rbr{\intc{Ts} \T{u}^Q \varphi(x)\chi(1-s + u/T) \dd{u}}^2 }\right| \dd{x}\dd{s} }{C \Theta^2  T^{-1} b}
	\]
	\[
		\leq \frac{ C_2 \intc{1} \intr  \rbr{\intc{Ts} \T{u}^Q \varphi(x)(\chi(1-s) + \chi(1-s + u/T)) \dd{u} }  \rbr{\intc{Ts} \T{u}^Q \varphi(x) |\chi(1-s) - \chi(1-s + u/T)| \dd{u} } \dd{x} \dd{s}}{ b T^{-1} }.
	\]
It is easy to check that $\intc{\cdot} \T{u}^Q \varphi(x) |\chi(1-s) - \chi(1-s + u/T)| \dd{u} \leq C_3$. Next we substitute $u\rightarrow Tu$, use definition of $\nu$ and assumption (A7) to obtain
\begin{multline*}
		\frac{|\Lambda_{2b}(T,\nu) - \Lambda_{2c}(T,\nu) |}{ H(T,\Theta,b) } \leq C_4 b^{-1} T^2 \intc{1} \intr  \intc{s} \T{Tu}^Q \varphi(x)(\chi(1-s) + \chi(1-s + u)) \dd{u} \dd{x}  \dd{s} = \\ C_5 b^{-1}T^2 \intc{1} \intr \T{Tu}^Q \varphi(x) \int_{u}^1 (\chi(1-s) + \chi(1-s + u)) \dd{s} \dd{x}  \dd{u}	\leq C_5 T^2 \intc{1} \intr \T{Tu}^Q \varphi(x)\dd{x} \dd{u} \rightarrow 0.
\end{multline*}
In similar way one can upper-bound $\frac{\Lambda_1(T,\nu) - H_1(T,\Theta, b) }{H(T,\Theta, b)}$ which concludes the proof. 
\end{proof}

We need also a method of estimating suprema of processes. Let us denote the set of dyadic rationals
\[
	D_k = \cbr{i/2^k: i\in \cbr{0,1,\ldots,2^k}}.
\]
\begin{lem}\label{lem:suprema}
	Let $x:[0,1]\rightarrow \R$ be a c\'adl\a'g function then
	\[
		\sup_{t\in[0,1]} |x(t)| \leq 2 \sum_{k=1}^\infty L_k(x) + |x(1)|,
	\]
	where 
	\[
		L_k(x) = \max_{\substack{r,s,t \in D_k \\s-r=t-s=2^{-k}} } m_{rst}(x),
	\]
	and $m_{rst}(x) = |x(s) - x(r)| \wedge |x(t) - x(s)|$.
\end{lem}
The proof is standard; the reader is referred to \cite[Section 10]{Billingsley:1999cl}. Now we proceed to the proof of exponential tightness. We start with \eqref{eq:exp-tight1}. Using Lemma \ref{lem:suprema} we write
\[
	\pr{\sup_{t\in[0,1]} |x_T(t)| \geq \eta } \leq \pr{2\sum_{k=1}^{\infty} L_k(x_T) \geq \eta/2 } + \pr{|x_T(1)|\geq \eta/2}.
\]
Therefore \eqref{eq:exp-tight1} will be shown once we obtain 
\[
	\lim_{\eta \rightarrow +\infty} \limsup_{T\rightarrow +\infty} T^{-\alpha} \log \pr{|x_T(1)|\geq \eta}= -\infty,\quad \lim_{\eta \rightarrow +\infty} \limsup_{T\rightarrow +\infty} T^{-\alpha} \log \pr{ \sum_{k=1}^{\infty} L_k(x_T) \geq \eta}  = -\infty. 
\]
The first one follows from \cite[Theorem 5.1]{Hong:2005aa}. To prove the second one we set 
\begin{equation}
	\epsilon_1 := (1-\alpha)/50, \quad \epsilon_2 := (1-\alpha)/70, \quad \theta := 2^{-\epsilon_1}, \label{eq:def-epsilons}
\end{equation}
and write
\begin{multline*}
	\pr{\sum_{k=1}^{\infty}L_k(x_T) \geq \eta} \leq \sum_{k=1}^{\infty} \pr{L_k(x_T) \geq \theta^k \eta_\theta}\\ \leq \sum_{k=1}^{\infty} 2^k \max_{i\in \cbr{1,2,\ldots,2^k}}\pr{|x_T(i2^{-k}) - x_T((i-1)2^{-k})| \geq \theta^k \eta_\theta},
\end{multline*}
where $\eta_\theta = \eta (1-\theta)/\theta$. We denote also $K_T := \frac{1-\alpha}{8 \log(2\theta)} \log T$ and $k_T := \log T$ and split the sum
\[
	\pr{\sum_{k=1}^{\infty}L_k(x_T) \geq \eta} \leq \sum_{k=1}^{K_T} \ldots + \sum_{k=K_T}^{k_T} \ldots + \sum_{k=k_T}^{\infty} \ldots =: I(T) + II(T) + III(T).
\]
\paragraph*{Estimation of $III(T)$} ~\\
First we will estimate the probability in the sum above. For $k\in \mathbb{N}$ we denote $\delta_k := 2^{-k}, l_k := 2^k$ and take any $u_1, u_2 \in \R_+$ such that $u_2 - u_1 = \delta_k$ and by $x$ total fluctuation (as in Lemma \ref{lem:estimate-xyz}). We have %// Using the same technique as in the in Corollary \ref{cor:increaments} and \eqref{eq:delta-v-branching} we can prove that
\begin{multline*}
	A_k := \lap{l_k(x(u_2) - x(u_1))} = \exp\cbr{\intc{u_2} \intr u^B_{\Psi_k}(x,u_2 - t,t) \dd{x} \dd{t}  } \\ = \exp \cbr{\intc{u_2} \intr \intc{t} \T{t-s}^Q \sbr{v^B_{\Psi_k}(\cdot,u_2 - s,s)^2}(x) \dd{s} \dd{x} \dd{t} },
\end{multline*}
where $\Psi_k(x,t) = l_k \varphi(x) \mathbf{1}_{[u_1, u_2]}(t)$. One must be aware that in the above equation we go slightly beyond the scope of Proposition \ref{prop:laplace-super} and Lemma \ref{lem:differentus}. However let us notice that all functions above are analytic as functions of complex parameter $l_k$.  We understand $u^S_{\Psi_k}$ and $v^S_{\Psi_k}$ as the analytic extension of the definitions in Section \ref{sec:one-particle}. Using assumption (A7) and the Fubini theorem we get
%Lemma \ref{lem:equation-super}  because of a special form of $\Psi_k$. Using the Morera theorem, the representation from Remark \ref{rem:representation} and the reasoning below one can check that both sides are analytic as a function of $theta_k$ (i.e. when we treat $\theta_k$ as a parameter). By Lemma \ref{lem:equation-super} both sides are equal for $\theta< Q^S_0$ hence they also agree for $\theta_k = 2^k$. 
\[
	\log A_k \leq  c \intc{u_2} \intr v^B_{\Psi_k}(x,u_2 - s,s)^2 \dd{x} \dd{s}.
\]
We are going to estimate the right-hand side. Let us notice that by Lemma \ref{lem:comparison} it is sufficient to prove the estimation of $\intc{u_2} \intr v^S_{C^{-1}\Psi_k}(x,u_2 - s,s)^2 \dd{x} \dd{s}$.  We denote $v_k(x,s):= v^B_{C^{-1}\Psi_k}(x,u_2-s,s)$. Equation \eqref{eq:w-integral-super} writes as
\[
	v_k(x,t) = \intc{t} \T{t-s}^Q \sbr{l_k \varphi(\cdot) \mathbf{1}_{[0,\delta_k]}(s) + v_k^2(\cdot,s) }(x) \dd{s}. 
\]
We are going to estimate $\norm{v_k(\cdot,t)}{2}{}$. Using the representation \eqref{eq:rep} (we use analytic extensions again and skip $Vq$ to make calculations trackable) we have 
\begin{equation*}
		v_k(x,t) = \sum_{k=0}^{\infty } F^{*n}_k(x,t),
\end{equation*}
where  $F_k^{*1}(x,t) = \intc{t} \T{t-s}^Q \sbr{l_k \varphi(\cdot) \mathbf{1}_{[0,\delta_k]}(s)}(x) \dd{s}$. Let us recall $Q'$ from assumption (A6); we will prove that%It is easy to check that the sum converges when $\varphi$ is small. We are going to prove that for $\varphi$ small enough we have 
\[
 \norm{F^{*n}_k(\cdot,t)}{2}{} \leq 2^{-n} e^{-(Q' t)/2}.
\]
For $n=1$ we have and $t>\delta_k$ using assumption (A6) and the generalised Minkowski inequality
\begin{multline*}
	\norm{F_k^{*1}(\cdot,t)}{2}{} \leq \norm{ \T{t-\delta_k}^Q \intc{t} \T{t-s}^Q \sbr{l_k \varphi(\cdot) \mathbf{1}_{[0,\delta_k]}(s)}(x) \dd{s}}{2}{} \leq l_k e^{-Q'(t-\delta_k)} \norm{ \intc{\delta_k} \T{s}^Q {\varphi(x) } \dd{s}}{2}{} \leq \\ l_k e^{-Q'(t-\delta_k)}\intc{\delta_k}  \norm{  \T{s}^Q {\varphi(x) } }{2}{} \dd{s} \leq l_k e^{-Q'(t-\delta_k)}\intc{\delta_k}  \norm{ {\varphi } }{2}{} \dd{s} = e^{-Q'(t-\delta_k)} \norm{ {\varphi } }{2}{} \leq  2^{-1}e^{-(Q't)/2}.
\end{multline*}
We decrease $\varphi$ is necessary - see Remark \ref{rem:smaller}. For $t<\delta_k$ it is easy to check that $\norm{F_k(\cdot,t)}{2}{}$ is even smaller. Using Lemma \ref{lem:hong-infty} we have $\norm{F^{*n}_k}{\infty}{} \leq  C^{-n}$ for any constant $C$ (possibly decreasing $\varphi$ once more). For $n>1$ we estimate using the induction argument together with the (generalised) Minkowski inequality
	\begin{multline*}
		\norm{F^{*n}_k(\cdot,t)}{2}{} = \norm{\sum_{j=1}^{n-1} \rbr{F^{*j}_k * F^{*(n-j)}_k} (\cdot,t)}{2}{} %\leq \sum_{j=1}^{n-1}\norm{ F_k(x,t)^{*j} * F_k(x,t)^{*(n-j)}}{2}{} \leq \\
\\	\leq	2 \sum_{j=1}^{\lceil n/2 \rceil}\norm{\intc{t} \T{t-s}^Q \sbr{  F^{*j}_k(\cdot,s)  F^{*(n-j)}_k(\cdot,s) }(x)\dd{s} }{2}{} \leq 2 \sum_{j=1}^{\lceil n/2 \rceil}\norm{ \intc{t} \T{t-s}^Q {F_k^{*j}(\cdot,s)} \dd{s} }{2}{} \norm{ F^{*(n-j)}_k}{\infty}{} \\
		\leq 2 \sum_{j=1}^{\lceil n/2 \rceil}\norm{ \intc{t} \T{t-s}^Q \sbr{F_k^{*j}(\cdot,s)} \dd{s} }{2}{} C^{-(n-j)} \leq  2 \sum_{j=1}^{\lceil n/2 \rceil} C^{-(n-j)} \intc{t} \norm{  \T{t-s}^Q \sbr{F_k^{*j}(\cdot,s)}  }{2}{} \dd{s}.
	\end{multline*}
	We can now use assumption (A6), the induction hypothesis and choose suitable $C>0$ to get  
	\begin{multline*}
	\norm{F_k(\cdot,t)^{*n}}{2}{} \leq	2 \sum_{j=1}^{\lceil n/2 \rceil} C^{-(n-j)} \intc{t} e^{-Q'(t-s) }2^{-j}e^{-(Q's)/2}  \dd{s} \\ \leq  e^{-(Q't)/2} \sum_{j=1}^{\lceil n/2 \rceil} Q' C^{-(n-j)}  2^{-j}  \leq 2^{-n} e^{-(Q't)/2}.
	\end{multline*}
It is now obvious that $\norm{v(\cdot,t)}{2}{} = \norm{\sum_{n=1}^{\infty}F(\cdot,t)^{*n}}{2}{} \leq  e^{-(Q't)/2}$ and the estimate does not depend on $k$, so neither does $A_k$. The Chebyshev inequality yields
%Recall now equations \eqref{eq:v-once-again} and \eqref{eq:v-once-again2} (we take $\chi = \mathbf{1}_{[0,2^{-k}]}$) for $t>2^{-k}$ the function $v_\theta$ is non-increasing (it is easy to check using the same expansion as above). Hence for any $|t-s|\leq \delta_k$
\begin{equation*}
	\pr{x(u_2) - x(u_1)\geq \lambda} \leq \frac{ {A_k} }{\exp (l_k \lambda)} \leq C_2 \exp\rbr{-2^k \lambda}.
\end{equation*}
It is easy to derive an analogous estimate for $\pr{x(u_1) - x(u_2)\geq \lambda}$. Employing these to $III(T)$  we get
\[
	III(T) \leq 2C_2 \sum_{k=k_T}^{\infty} 2^k \exp\rbr{- 2^{k-k_T} T^{(1+\alpha)/2} \theta^k \eta_\theta } =2 C_2 \sum_{k=k_T}^{\infty}  \exp\rbr{k \ln 2 - (2\theta)^{k-k_T} \theta^{k_T} T^{(1+\alpha)/2} \eta_\theta }
\]
The choice of $\theta$ yields $\theta^{k_T} = T^{-\epsilon_1}$. For $T$'s large enough (depending on $\eta_\theta$) and certain $C>0$ we have $k \ln 2 - (2\theta)^{k-k_T} T^{-\epsilon_1+ (1+\alpha)/2} \eta_\theta \leq - k C T^{\alpha} \eta_\theta$, hence
%Recall that $\theta > 1/\sqrt{2}$ hence there exists $c_1$ such that $(2\theta)^k \geq c_1 k$ and for $T$'s large enough 
\[
		III(T) \leq \sum_{k=k_T}^{\infty}  \exp\rbr{- k C T^{\alpha} \eta_\theta } \leq   \frac{\exp\rbr{- k_T C T^{\alpha} \eta_\theta }}{1 - \exp\rbr{- k_T C T^{\alpha} \eta_\theta } }.
\]
It is now straightforward to check that $\lim_{\eta \rightarrow +\infty} \limsup_{T\rightarrow +\infty} T^{-\alpha} \log III(T) = -\infty$.

\paragraph*{Estimation of $I(T)$} We will use Lemma \ref{lem:estimate-xyz} with $\epsilon := \frac{1-\alpha}{4}$ and $\theta'=l_T$ given by
\[
	l_T := \frac{\eta_\theta}{2 c} (2\theta)^k T^{(\alpha-1)/2},
\]
where $c$ is the same as in the lemma. It is straightforward to check that for $T$'s large enough (depending on $\eta_\theta, \theta$ and $c$) and for $k<K_T$ we have $l_T \leq T^{(\alpha-1)/4}$. Consequently, by Lemma \ref{lem:estimate-xyz} and the Chebyshev inequality we have 
\[
	\pr{x_T(i2^{-k}) - x_T((i-1)2^{-k}) \geq \theta^k \eta_\theta} \leq \exp \rbr{l_T^2 c T 2^{-k} - l_T T^{(1+\alpha)/2} \theta^k \eta_\theta  } = \exp\rbr{-\frac{\eta_\theta^2}{4c} (2\theta^2)^k T^\alpha}.
\]
An analogous inequality for $\pr{x_T((i-1)2^{-k}) - x_T(i2^{-k}) \geq  \theta^k \eta_\theta}$ also holds. Consequently
\[
	I(T) \leq 2 \sum_{k=1}^{K_T} 2^k \exp\rbr{-\frac{\eta_\theta^2}{4c} (2\theta^2)^k T^\alpha} =  2\sum_{k=1}^{K_T} \exp\rbr{k \ln 2-\frac{\eta_\theta^2}{4c} (2\theta^2)^k  T^\alpha}.
\]
Recalling \eqref{eq:def-epsilons} it is easy to check that $\theta> 1/\sqrt{2}$ hence $(2\theta^2)^k > c_1 k $. Finally we get
\[
	I(T) \leq \sum_{k=1}^{K_T} \exp\rbr{k \ln 2- c_1 k \frac{\eta_\theta^2}{4c} T^\alpha} \leq \frac{\exp\rbr{\ln 2 - c_1 \frac{\eta_\theta^2}{4c} T^\alpha}}{1 -\exp\rbr{\ln 2 - c_1 \frac{\eta_\theta^2}{4c} T^\alpha} }
\]
It is now straightforward to check that $\lim_{\eta \rightarrow +\infty} \limsup_{T\rightarrow +\infty} T^{-\alpha} \log I(T) = -\infty$.\\
\paragraph*{Estimation of $II(T)$} By \eqref{eq:def-epsilons} one checks that  $(7\alpha+1)/8 - \epsilon_2 -\epsilon_1 \geq \alpha$. We will use Lemma \ref{lem:estimate-xyz} with some $\epsilon < (3 - 3\alpha)/8-\epsilon_2$ and $\theta'=l_T$ given by
\[
	l_T =\frac{\eta_\theta}{ 2c } T^{(3\alpha-3)/8-\epsilon_2},
\]
where $c$ is the same as in the lemma. For large $T$ (depending on $\eta_\theta$ and $c$) we have $l_T \leq T^{-\epsilon}$. By Lemma \ref{lem:estimate-xyz}  and the  Chebyshev inequality we get 
\begin{multline*}
	\pr{x_T(i2^{-k}) - x_T((i-1)2^{-k}) \geq \theta^k \eta_\theta} \leq \exp \rbr{l_T^2 c T 2^{-k} - l_T T^{(1+\alpha)/2} \theta^k \eta_\theta  } = \\
	\exp \rbr{ \frac{\eta^2_\theta}{4 c}T^{(3\alpha+1)/4 - 2\epsilon_2 } 2^{-k} - \frac{\eta^2_\theta}{2 c} T^{(7\alpha+1)/8 - \epsilon_2} \theta^k} \leq \exp \rbr{ \frac{\eta^2_\theta}{4 c}T^{(7\alpha+1)/8-2\epsilon_2}  - \frac{\eta^2_\theta}{2 c} T^{(7\alpha+1)/8 - \epsilon_2 -\epsilon_1} },
\end{multline*}
where the last estimate follows by $\theta^{k_T} \geq T^{-\epsilon_1}$ and $2^{-k} \leq 2^{-K_T} \leq T^{(\alpha-1)/8}$. An analogous estimate holds also for $\pr{x_T((i-1)2^{-k}) - X_T(i2^{-k}) \geq \theta^k \eta_\theta}$. Putting these together we write
\[
	II(T) \leq C_1 T \exp \rbr{ \frac{\eta^2_\theta}{4 c}T^{\rbr{7\alpha+1}/{8}-2\epsilon_2}  - \frac{\eta^2_\theta}{2 c} T^{(7\alpha+1)/8 - \epsilon_2 -\epsilon_1} }.
\]
It is now straightforward to check that $\lim_{\eta \rightarrow +\infty} \limsup_{T\rightarrow +\infty} T^{-\alpha} \log I(T) = -\infty$.\\
% Now it is easy to check that
% \[
% 	\limsup_{T\rightarrow +\infty} T^{-\alpha} \log II(T) \leq -\frac{\eta^2_\theta}{2 c}
% \]
This end the proof of \eqref{eq:exp-tight1}. Now we turn to \eqref{eq:exp-tight2}. For any $\tau \in ST(x_T)$ we have 
\[
	\sup_{t\leq \delta} |x_T((\tau+t)\wedge 1) - x_T(\tau)| \leq w(x_T, \delta),
\]
where $w$ is the modulus of continuity \eqref{eq:modulus}. Using this fact together with \cite[Theorem 7.4]{Billingsley:1999cl} we get 
\[
	\sup_{\tau \in ST(x_T)}\pr{\sup_{t\in[0,\delta]} |x_T((\tau+t)\wedge 1) - x_T(\tau)|\geq \lambda} \leq \sup_{s\in[0,1-\delta]} \delta^{-1}\pr{\sup_{t\in[0,\delta]} |x_T(s+t) - x_T(s)|\geq \lambda/3}. 
\]
To prove \eqref{eq:exp-tight2} it is enough to prove that for any $\lambda>0$ there is
\begin{equation}
	\lim_{\delta \rightarrow 0}  \limsup_{T\rightarrow +\infty} T^{-\alpha}\sup_{s\in[0,1-\delta]} \log \pr{\sup_{t\in[0,\delta]} |x_T(s+t) - x_T(s)|\geq \lambda}  = -\infty. \label{eq:exp-tight3} 
\end{equation}
The following proof mimics the proof of \eqref{eq:exp-tight1} but is slightly more technically elaborated. By Lemma \ref{lem:suprema} we have
\[
	\pr{\sup_{t\in[0,\delta]} |x_T(t+s) - x_T(s)| \geq \lambda } \leq \pr{2 \sum_{k=1}^{\infty} L^{\delta,s}_k(x_T) \geq \lambda/2 } + \pr{|x_T(s+\delta) - x_T(s)|\geq \lambda/2}
\]
where $L^{\delta,s}_k$ is defined analogously to $L_k$ but on the interval $[s,s+\delta]$. Finally, \eqref{eq:exp-tight3} will be shown once we have proved that 
\[
	\lim_{\delta \rightarrow 0}  \limsup_{T\rightarrow +\infty} T^{-\alpha}\sup_{s\in[0,1-\delta]} \log \pr{|x_T(s+\delta) - x_T(s)|\geq \lambda}  = -\infty,
\]
\[
	\lim_{\delta \rightarrow 0}  \limsup_{T\rightarrow +\infty} T^{-\alpha}\sup_{s\in[0,1-\delta]} \log \pr{\sum_{k=1}^{\infty} L^{\delta,s}_k(x_T) \geq \lambda}  = -\infty.
\]
The first convergence can be  obtained by application of Lemma \ref{lem:estimate-xyz} with $\theta'=\frac{\lambda}{2c\delta}T^{(\alpha-1)/2}$ and the Chebyshev inequality, namely $\pr{|x_T(s+\delta) - x_T(s)|\geq \lambda} \leq \exp\rbr{-\frac{\lambda}{4c\delta}T^\alpha}$. To prove the second we recall \eqref{eq:def-epsilons} 
% \begin{equation}
% 	\epsilon_1 := (1-\alpha)/50, \quad \epsilon_2 := (1-\alpha)/70, \quad  \theta := 2^{-\epsilon_1}, \label{eq:def-epsilons2}
% \end{equation}
and write
\begin{multline*}
	\pr{\sum_{k=1}^{\infty}L^{\delta,s}_k(x_T) \geq \eta} \leq \sum_{k=1}^{\infty} \pr{L_k^{\delta,s}(x_T) \geq c_\theta \theta^k \eta}\\
	\leq \sum_{k=1}^{\infty} 2^k \max_{i\in \cbr{1,2,\ldots,2^k}}\pr{|x_T(i2^{-k}) - x_T((i-1)2^{-k})| \geq c_\theta \theta^k \eta}
\end{multline*}
where $\lambda_\theta = \lambda (1-\theta)/\theta$. We denote also $K_T := \frac{1-\alpha}{8 \log(2\theta)} \log T$ and $k_T := \log (\delta T)$ then
\[
	\pr{\sum_{k=1}^{\infty} L^{\delta,s}_k(x_T) \geq \eta} \leq \sum_{k=1}^{K_T} \ldots + \sum_{k=K_T}^{k_T} \ldots + \sum_{k=k_T}^{\infty} \ldots =: I(T) + II(T) + III(T).
\]
\paragraph*{Estimation of $III(T)$}
Following the same lines of reasoning as in the previous section we  arrive at
\begin{multline*}
	III(T) \leq C \sum_{k=k_T}^{\infty} 2^k \exp\rbr{-2^{k-k_T} \delta^{-1} T^{(1+\alpha)/2} \theta^k \lambda_\theta }\\ \leq C \sum_{k=k_T}^{\infty}  \exp\rbr{k \ln 2 - (2\theta)^{k-k_T} \theta^{k_T} \delta^{-1} T^{(1+\alpha)/2} \lambda_\theta }.
\end{multline*}
The choice of $\theta$ implies $\theta^{k_T} = (T\delta)^{-\epsilon_1}$. For $T$'s large enough (depending on $\lambda_\theta$ and $\delta$) and certain $C>0$ we have $k \ln 2 - (2\theta)^{k-k_T} \theta^{k_T} \delta^{-1} T^{(1+\alpha)/2} \lambda_\theta \leq -k C \delta^{-1-\epsilon_1} T^\alpha \lambda_\theta$, hence
\[
		III(T) \leq \sum_{k=k_T}^{\infty}  \exp\rbr{-k C \delta^{-1-\epsilon_1} T^\alpha \lambda_\theta} \leq   \frac{\exp\rbr{-k_T C \delta^{-1-\epsilon_1} T^\alpha \lambda_\theta}}{1 - \exp\rbr{-k_T C \delta^{-1-\epsilon_1} T^\alpha \lambda_\theta} }.
\]
It is now straightforward to check that for any $\lambda>0$ we have $\lim_{\delta \rightarrow 0} \limsup_{T\rightarrow +\infty} T^{-\alpha} \log III(T) = - \infty$.

\paragraph*{Estimation of $I(T)$} We use Lemma \ref{lem:estimate-xyz} with $\epsilon = \frac{1-\alpha}{4}$ and $\theta'=l_T$ ($c$ is given by the lemma)
\[
	l_T = \frac{\lambda_\theta}{2 \delta c} (2\theta)^k T^{(\alpha-1)/2}.
\]
It is straightforward to check that for $T$'s large enough (depending on $\lambda_\theta, \delta, \theta$ and $c$) for any $k<K_T$ we have $l_T \leq T^{(\alpha-1)/4}$. Consequently, by Lemma \ref{lem:estimate-xyz} and the Chebyshev inequality we have 
\begin{multline*}
	\pr{x_T(s+ i\delta 2^{-k}) - x_T(s+ (i-1) \delta 2^{-k}) \geq  \theta^k \lambda_\theta} \\
	\leq \exp \rbr{l_T^2 c \delta T 2^{-k} - l_T T^{(1+\alpha)/2} \theta^k \lambda_\theta  } = \exp\rbr{-\frac{\lambda_\theta^2}{4\delta c} (2\theta^2)^k T^\alpha}.
\end{multline*}
An analogous inequality holds also for $\pr{x_T(s+ (i-1)\delta 2^{-k}) - x_T(s+ i \delta 2^{-k}) \geq  \theta^k \lambda_\theta}$. Consequently
\[
	I(T) \leq 2 \sum_{k=1}^{K_T} 2^k \exp\rbr{-\frac{\lambda_\theta^2}{4\delta c} (2\theta^2)^k T^\alpha} =  \sum_{k=1}^{K_T} \exp\rbr{k \ln 2-\frac{\lambda_\theta^2}{4\delta c} (2\theta^2)^k  T^\alpha}.
\]
We know that $\theta> 1/\sqrt{2}$ hence $(2\theta^2)^k > c_1 k $ for certain $c_1>0$. Finally we get
\[
	I(T) \leq \sum_{k=1}^{K_T} \exp\rbr{k \ln 2- c_1 k \frac{\lambda_\theta^2}{4\delta c} T^\alpha} \leq \frac{\exp\rbr{\ln 2 - c_1 \frac{\lambda_\theta^2}{4\delta c} T^\alpha}}{1 -\exp\rbr{\ln 2 - c_1 \frac{\lambda_\theta^2}{4\delta c} T^\alpha} }.
\]
It is now straightforward to check that for any $\lambda>0$ we have $\lim_{\delta \rightarrow 0} \limsup_{T\rightarrow +\infty} T^{-\alpha} \log I(T) = - \infty$.
% \[
% 	\limsup_{T \rightarrow +\infty} T^{-\alpha} \log I(T) \leq - c_1 \frac{\lambda_\theta^2}{4\delta c}
% \]
% Taking $\delta$ small we can make the last quantity arbitrary negative.

\paragraph*{Estimation of $II(T)$}  We apply Lemma \ref{lem:estimate-xyz} with some $\epsilon < (3 - 3\alpha)/8-\epsilon_2$ and $\theta'=l_T$ given by
\[
	l_T =\frac{\lambda_\theta}{ 2\delta c } T^{(3\alpha-3)/8-\epsilon_2},
\]
where $c$ is given by the lemma. For $T$'s large enough (depending on $\lambda_\theta, \theta, \delta$ and $c$)  $l_T \leq T^{-\epsilon}$ hence Lemma \ref{lem:estimate-xyz} and the  Chebyshev inequality yield 
\begin{multline*}
	\pr{x_T(s+ i\delta 2^{-k}) - x_T(s+(i-1)\delta 2^{-k}) \geq \theta^k \lambda_\theta} \leq \exp \rbr{l_T^2 c T \delta 2^{-k} - l_T T^{(1+\alpha)/2} \theta^k \lambda_\theta  } = \\
	\exp \rbr{ \frac{\lambda_\theta^2}{4 \delta  c}T^{(3\alpha+1)/4 - 2\epsilon_2 } 2^{-k} - \frac{\lambda_\theta^2}{2 \delta  c} T^{(7\alpha+1)/8 - \epsilon_2} \theta^k} \\\leq \exp \rbr{ \frac{\lambda_\theta^2}{4 \delta c}T^{\frac{7\alpha+1}{8}-2\epsilon_2}  - \frac{\lambda_\theta^2 }{2\delta^{1+\epsilon_1}  c} T^{(7\alpha+1)/8 - \epsilon_2 -\epsilon_1} },
\end{multline*}
where in the last estimation we used the fact that $\theta^{k_T} \approx (\delta T)^{-\epsilon_1}$ and $2^{-k} \leq 2^{-K_T} \leq T^{(\alpha-1)/8}$. An analogous estimate holds also for $\pr{x_T(s+ (i-1)\delta 2^{-k}) - x_T(s+i\delta 2^{-k}) \geq  \theta^k \lambda_\theta}$. Hence
\[
	II(T) \leq C \delta T \exp \rbr{ \frac{\lambda_\theta^2}{4 \delta c}T^{\rbr{7\alpha+1}/{8}-2\epsilon_2}  - \frac{\lambda_\theta^2 c_\theta}{2 \delta^{1+\epsilon_1}c} T^{(7\alpha+1)/8 - \epsilon_2 -\epsilon_1} }
\]
It is now straightforward to check that for any $\lambda>0$ we have $\lim_{\delta \rightarrow 0} \limsup_{T\rightarrow +\infty} T^{-\alpha} \log II(T) = - \infty$.
% Now it is easy to check that
% \[
% 	\limsup_{T\rightarrow +\infty} T^{-\alpha} \log II(T) \leq -\frac{\lambda_\theta^2 }{2 \delta^{1+\epsilon_1} c}
% \]
% Again taking $\delta$ small enough we can make the limit arbitrary negative.

\bibliographystyle{plain} 
\bibliography{Resources/Library/branching}

\end{document}